\documentclass[a4paper,11pt,reqno]{amsart}

\usepackage[utf8]{inputenc}
\usepackage{tikz}
\usetikzlibrary{arrows}

\usepackage{amssymb,upref}
\usepackage[mathcal]{euscript}
\usepackage[cmtip,all]{xy}

\newcommand{\bydef}{:=}

\newcommand{\cA}{\mathcal{A}}
\newcommand{\cE}{{\mathcal E}}
\newcommand{\cI}{\mathcal{I}}
\newcommand{\cJ}{\mathcal{J}}

\newcommand{\cS}{{\mathcal S}}
\newcommand{\cT}{{\mathcal T}}
\newcommand{\cU}{{\mathcal U}}

\newcommand{\espan}[1]{\mathrm{span}\left\{#1\right\}}

\newcommand{\NN}{{\mathbb N}}
\newcommand{\FF}{\mathbb{F}}
\newcommand{\RR}{\mathbb{R}}
\newcommand{\id}{\mathrm{id}}

\DeclareMathOperator{\ann}{ann}

\newtheorem{theorem}{Theorem}[section]
\newtheorem{proposition}[theorem]{Proposition}
\newtheorem{lemma}[theorem]{Lemma}
\newtheorem{corollary}[theorem]{Corollary}

\theoremstyle{definition}
\newtheorem{df}[theorem]{Definition}

\numberwithin{equation}{section}

\theoremstyle{remark}
\newtheorem{remark}[theorem]{Remark}

\newenvironment{romanenumerate}
{\begin{enumerate}
 
 }{\end{enumerate}}

\usepackage{tikz}
\usetikzlibrary{arrows}

\begin{document}

\title{On nilpotent evolution  algebras }

\author[Alberto Elduque]{Alberto Elduque$^{\star}$}
\thanks{$^{\star}$ Supported by the Spanish Ministerio de Econom{\'\i}a y Competitividad and 
Fondo Europeo de Desarrollo Regional (FEDER) MTM 2013-45588-C3-2-P, 
and by the Diputaci\'on General de Arag\'on -- Fondo Social Europeo (Grupo de Investigaci\'on de \'Algebra).
Part of this research was carried out while this author was visiting the Departamento de Matem\'aticas, Facultad de Ciencias, Universidad de
Chile, supported by the FONDECYT grant 1120844.}
\address{Departamento de Matem\'aticas e Instituto Universitario de Matem\'aticas y Aplicaciones,
Universidad de Zaragoza, 50009 Zaragoza, Spain
}
\email{elduque@unizar.es}

\author[Alicia Labra]{Alicia Labra$^{\star\star}$}
\thanks{$^{\star\star}$ Supported by FONDECYT 1120844.}
\address{Departamento de Matem\'aticas,
Facultad de Ciencias, Universidad de Chile.  Casilla 653, Santiago, Chile}
\email{alimat@uchile.cl}




\begin{abstract}

The type and several invariant subspaces related to the upper annihilating series of finite-dimensional 
nilpotent evolution algebras are introduced. These invariants can be easily computed from any natural basis.
Some families of nilpotent evolution algebras, defined in terms of a nondegenerate, symmetric, bilinear form 
and some commuting, symmetric, diagonalizable endomorphisms relative to the form, are explicitly constructed. 
Both the invariants and these families are used to review and complete the classification of nilpotent 
evolution algebras up to dimension five over algebraically closed fields.
\end{abstract}

\maketitle


\section{Introduction}

Evolution algebras were introduced in 2006 by Tian and Vojtechovsky, in their paper ``Mathematical
concepts of evolution algebras in non-Mendelian genetics " (see \cite{TV}).
Later on, Tian laid the foundations of evolution algebras in his monograph
\cite{T}. 

In some recent papers \cite{HA1,HA2}, a classification of the nilpotent evolution algebras up to dimension 
five has been given. However, there is a subtle point which has not been considered. When dealing with an extension of a 
nilpotent evolution algebra by a trivial ideal, one cannot fix a natural basis of the initial algebra, because
a natural basis of a quotient does not necessarily extend to a natural basis of the whole algebra. As a consequence, the
classifications in these papers are not complete. This also shows how tricky these algebras are.

The goal of this paper is the introduction of some new techniques for the study of evolution algebras, as well as the 
construction of several noteworthy families of nilpotent evolution algebras defined in terms of bilinear forms and symmetric 
endomorphisms. Using these tools, the classification of the nilpotent evolution algebras up to dimension five, over an algebraically
closed field of characteristic not two, is obtained without much effort, although the number of possibilities in dimension
five is quite high and indicates the difficulty of this problem for higher dimension.

\smallskip

Let us first recall the basic definitions.

An  \emph{evolution algebra} is an algebra $\cE$ containing a countable basis (as a vector space) 
$B=\{e_1, \ldots, e_n, \ldots,\}$ such that $e_ie_j=0$ for any $1\leq i\ne j\leq n$. 
A basis with this property is called a \emph{natural basis}. By its own definition, 
any evolution algebra is commutative.
In this paper we deal with finite dimensional evolution algebras.
Given a natural basis $B=\{e_1,\ldots,e_n\}$ of an evolution algebra $\cE$,
$e_i^2=\sum_{j=1}^n\alpha_{ij}e_j$ 
for some scalars $\alpha_{ij}\in \FF$, $1\leq i,j\leq n$. The matrix $A=\bigl(\alpha_{ij}\bigr)$ is 
the \emph{matrix of structural constants} of the evolution algebra $\cE$, relative to the natural basis $B$.

We recall next the definition of the graph and weighted graph attached to an evolution algebra (see 
\cite[Definition 2.2]{EL}). Our 
graphs are always directed graphs, and most of our algebras will be presented by means of their graphs.

Let $\cE$ be an evolution algebra with a natural basis $B=\{e_1,\ldots,e_n\}$ and matrix of 
structural constants $A=\bigl(\alpha_{ij}\bigr)$.
\begin{itemize}
\item The graph $\Gamma(\cE,B)=(V,E)$, with $V=\{1,\ldots,n\}$ and $E=\{(i,j)\in V\times V: \alpha_{ij}\ne 0\}$, is called the \emph{graph attached to the evolution algebra $\cE$ relative to the natural basis $B$}.
\item The triple $\Gamma^w(\cE,B)=(V,E,\omega)$, with $\Gamma(\cE,B)=(V,E)$ and where $\omega$ is 
the map $E\rightarrow \FF$ given by $\omega\bigl((i,j)\bigr)=\alpha_{ij}$, is called the 
\emph{weighted graph attached to the evolution algebra $\cE$ relative to the natural basis $B$}.
\end{itemize}

\smallskip

The paper is organized as follows.
 In Section \ref{se:Krull-Schmidt theorem_and_indescomposability} we prove a general Krull-Schmidt Theorem 
for nonassociative algebras, which has its own independent interest. It shows that it is enough
to classify indecomposable algebras. Also, using the annihilator of an algebra, we give some results which are useful to check the decomposability of a finite-dimensional algebra. 
See Lemma \ref{lem:subalgebraann}, Corollaries \ref{cor:algebracuadann} and \ref{cor:dimensionann}.
 
In Section \ref{se: upperseries} we define the upper annihilating series of an arbitrary nonassociative algebra,
and then the type of a finite-dimensional nilpotent algebra. This allows us later on to split the classification of 
nilpotent evolution algebras according to their types.
 
 In Section \ref{se: somefamilies-nilpotent_evol} we study some families of nilpotent evolution algebras, defined in terms of a 
 nondegenerate, symmetric, bilinear form and some commuting, symmetric, diagonalizable endomorphisms relative to the bilinear
 form.

Section \ref{se:classificationdimfour} is devoted to the classification of the indecomposable nilpotent evolution algebras
of dimension up to four, over an algebraically closed field of characteristic not two. All such algebras lie in one of the families 
studied in Section \ref{se: somefamilies-nilpotent_evol}, so this classification is done very quickly. Our list
includes two algebras not considered in \cite{HA1}.

Finally, in Section \ref{se:classificationdimfive}
 we classify all the indecomposable nilpotent evolution algebras of dimension five,  
 over an algebraically closed field of characteristic not two. About half of the algebras in this classification belong
 to one of the families in Section \ref{se: somefamilies-nilpotent_evol}. For the remaining algebras, some ad hoc arguments
 are needed. The results in \cite{HA2} miss most of the algebras in our classification.
 

\medskip

\section{A Krull-Schmidt Theorem for nonassociative algebras}\label{se:Krull-Schmidt theorem_and_indescomposability}

Given a nonassociative (i.e. not necessarily associative) algebra $\cA$ over a field $\FF$, its multiplication algebra $M(\cA)$ is the subalgebra of 
End$_{\FF}(\cA)$ generated by the left and right multiplications by elements in $\cA$:
$$ 
M(\cA)\bydef \; \mbox{alg} \langle L_x, R_x : x \in \cA \rangle, 
$$
where $L_x: y \mapsto xy$, $R_x: y \mapsto yx$.

By its own definition,  $\cA$ is a left module for the associative algebra $M(\cA)$ and the ideals of $\cA$ are precisely the submodules of $\cA$ as an $M(\cA)$-module.

\begin{df}
An algebra $\cA$ is said to be \emph{indescomposable} (resp. \emph{descomposable}) if it is so as an  
$M(\cA)$-module. 
That is, $\cA$ is descomposable  if there are nonzero ideals
 $\cI$ and  $\cJ$ such that $\cA = \cI \oplus \cJ$. Otherwise, it is indecomposable.
\end{df}

 In \cite{CSV} the word (ir)reducible is used instead of (in)descomposable. For evolution algebras, indescomposability is related to connectedness (see \cite[Proposition  2.8]{EL}).

\begin{theorem}\label{teo:Krull-Schmidt Theorem}
 Let $\cA$ be an algebra which is a module of finite length for $M(\cA)$ (this is always the case if  $\cA$ is finite-dimensional). Then $\cA$ decomposes as a finite direct sum of indescomposable ideals. 
 
 Moreover, if $\cA = \cI_1 \oplus \cdots \oplus \cI_n = \cJ_1 \oplus \cdots \oplus \cJ_m $ with $\cI_i, \cJ_j$ indecomposable for  all $ i = 1, \ldots, n$ and
  $ j = 1, \ldots, m$, then $ n = m$ and there is a permutation $\sigma \in S_n$ such that $\cI_i$ is isomorphic (as an algebra) to $\cJ_{\sigma(i)}$ for all $ i = 1, \ldots, n$. 
\end{theorem}
\begin{proof}
The version of the classical  Krull-Schmidt Theorem for modules proved in (\cite[Chapter V  \S 13]{J}) shows that $\cA$ is a finite direct sum of indescomposable ideals and that if $\cA = \cI_1 \oplus \cdots \oplus \cI_n = \cJ_1 \oplus \cdots \oplus \cJ_m $ with the  $\cI_i's$ and $ \cJ_j's$ indescomposable ideals, then $ n = m$ and, after a suitable reordering of the ideals, $\cA = \cI_1 \oplus \cdots \oplus \cI_k \oplus \cJ_{k+1}\oplus \cdots \oplus \cJ_n  $  
for all $ k = 0, \ldots, n$.

Then, for any  $k = 1, \ldots, n$, we have both $\cA = \cI_1 \oplus \cdots \oplus \cI_k \oplus \cJ_{k+1}\oplus \cdots \oplus \cJ_n  $ and $\cA = \cI_1 \oplus \cdots \oplus \cI_{k-1}\oplus \cJ_k \oplus \cJ_{k+1}\oplus \cdots \oplus \cJ_n, $ so both $\cI_k$ and $\cJ_k$ are isomorphic to the quotient 
$ \cA / (\cI_1 \oplus \cdots \oplus \cI_{k-1}\oplus \cJ_{k+1}\oplus \cdots \oplus \cJ_n)$. 
\end{proof}

\begin{remark}
In the proof above, from the fact that $\cI_k$ and $\cJ_k$ are isomorphic as $M(\cA)$-modules, it does not follow that they are isomorphic as algebras. 
Hence the explicit version of the  classical  Krull-Schmidt Theorem used here is essential.

For Bernstein algebras, a similar argument appears in \cite{CM}.
\end{remark}

The previous result shows that it is enough to classify indecomposable algebras. One has to take into account that any quotient, and hence any direct summand (as ideals), of an evolution algebra is itself an evolution algebra \cite[Lemma 2.9]{EL}.

\medskip
We finish this section with some useful tricks to check  descomposability. Recall 
that the annihilator  of an algebra $\cA$ is 
 $\ann(\cA)\bydef\{x\in\cA: x\cA=\cA x=0\}$.
 Any subspace of $\ann(\cA)$ is an ideal of $\cA$.

 \begin{lemma} \label{lem:subalgebraann}
  Let $\cA$ be an algebra over a field $\FF$, $\dim_{\FF}(\cA) > 1$. If $\cS$ is a proper subalgebra of $\cA$ such that $\cA = \cS + \ann(\cA)$, then $\cA$ is decomposable.
 \end{lemma}
 \begin{proof}
 If $\cA = \cS + \ann(\cA)$, then $\cA^2 = \cS^2 \subseteq \cS$, so $\cS$ is an ideal of $\cA$. Let $\cT$ be a subspace of $\ann(\cA)$ 
 such that $\cA =\cS \oplus \cT$. Then
 both $\cS$ and $\cT$ are ideals of $\cA$, so $\cA$ is decomposable.
 \end{proof}

 \begin{corollary}\label{cor:algebracuadann}
  Let $\cA$ be an algebra, $\dim_{\FF}(\cA) > 1$. If $\ann(\cA)$ is not contained in $\cA^2,  $ then $\cA$ is decomposable.
 \end{corollary}
 \begin{proof}
 We have that $\cA^2 \subsetneq \ann(\cA) + \cA^2$, so there is a proper subspace $\cS$ such that $\cA^2 \subseteq \cS$ and 
 $ \cA = \cS + \ann(\cA)$. Then $ \cS$ is an ideal of $\cA $ and the Lemma applies.
 \end{proof}

 \begin{corollary}\label{cor:dimensionann}
  Let $\cE$ be a finite-dimensional evolution algebra such that $\dim_{\FF}(\ann(\cE)) \geq \frac{1}{2} \dim_{\FF}(\cE)\geq 1$. Then $\cE$ is decomposable.
 \end{corollary}
 \begin{proof} 
 Let $B=\{e_1, \ldots, e_n,\}$  be a natural basis, ordered so that $e_1^2 = \ldots = e_r^2 = 0 $ and $e_{r+1}^2, \ldots, e_n^2 \neq 0. $
 Then $\ann(\cE)=\espan{e_1,\ldots,e_r }$ (\cite[Lemma 2.7]{EL}) and $\cE^2=\espan{e_{r+1}^2,\ldots,e_n^2}$. 
 Our hypotheses show that $n \leq 2r$. If $n<2r$, then
$\dim_{\FF}(\ann(\cE)) = r > n-r \geq \dim_{\FF}(\cE^2)$. Hence $\ann(\cE) \nsubseteq \cE^2$ and the previous Corollary applies. If $n=2r$ and $\ann(\cE)\nsubseteq \cE^2$, again Corollary \ref{cor:algebracuadann} applies. Finally, if $n=2r$ and $\ann(\cE)\subseteq \cE^2$, then $\cE^2=\espan{e_{r+1}^2,\ldots,e_n^2}$ equals $\ann(\cE)$ by dimension count, so the family $\{e_{r+1}^2,\ldots,e_n^2,e_{r+1},\ldots,e_n\}$ is another natural basis, and $\cE$ is the direct sum of the ideals $\cI_i=\espan{e_{r+i}^2,e_{r+i}}$, $i=1,\ldots,r$.
 \end{proof}

\medskip

\section{Upper annihilating series}\label{se: upperseries}

Given a nonassociative algebra $\cA$, we introduce the following sequences of subspaces:
\begin{align*}
\cA^{<1>} & = \cA,&  \cA^{<k+1>} & = \cA^{<k>}\cA;\\[-4pt]
\cA^1 & = \cA, & \cA^{k+1} &=\sum_{i=1}^{k}\cA^i\cA^{k+1-i}.
\end{align*}

\begin{df} An algebra $\cA$ is called
\begin{romanenumerate}
\item \emph{right nilpotent} if there exists $n\in \NN$ such that $\cA^{<n>} = 0$, and the minimal 
such number is called the \emph{index of right nilpotency};
\item \emph{nilpotent} if there exists $n\in \NN$ such that $\cA^n = 0$, and the minimal such number 
is called the \emph{index of nilpotency}.
\end{romanenumerate}
\end{df}

\begin{remark}
A commutative algebra is right nilpotent if and only if it is nilpotent (see \cite[Chapter 4, Proposition 1]{ZSSS}). This applies, in particular, to evolution algebras.
\end{remark}

\begin{df}
Let $\cA$ be an algebra. 
Consider the chain of 
ideals $\ann^i(\cA)$, $i\geq 1$, where:
\begin{itemize}
\item
 $\ann^1(\cA)\bydef\ann(\cA)\bydef\{x\in\cA: x\cA=\cA x=0\}$,
\item
$ \ann^i(\cA)$ is defined by $\ann^i(\cA) /\ann^{i-1}(\cA)\bydef\ann(\cA /\ann^{i-1}(\cA))$.
\end{itemize}
The chain of ideals:
\[
0=\ann^0(\cA)\subseteq \ann^1(\cA)\subseteq \cdots\subseteq \ann^r(\cA)\subseteq \cdots
\] 
is called the  \emph{the upper annihilating series}.
\end{df}

As for Lie algebras, a nonassociative algebra $\cA$ is nilpotent if and only if its upper annihilating series reaches $\cA$. 
That is, if there exists $r$ such that $\ann^r(\cA)=\cA$.

\medskip
\begin{df}
Let $\cA$ be a finite-dimensional nilpotent nonassociative algebra over a field $\FF$, and let  $r$ be the lowest natural number
with $\ann^r(\cA)=\cA$. The \emph{type} of $\cA$ is the sequence 
$[n_1, \ldots, n_r]$ such that for all $i = 1,\ldots, r$, $n_1 + \cdots + n_i = \dim_\FF(\ann^i(\cE))$. 
In other words,
\[
n_i =\dim_\FF(\ann(\cE/ann^{i-1}(\cE)))=\dim_\FF(\ann^i(\cE))-\dim_\FF(\ann^{i-1}(\cE)),
\] 
for all $i = 1,\ldots, r $.
\end{df}

\medskip
If $\cE$ is a nilpotent evolution algebra of type $[n_1, \ldots, n_r]$ and  $B=\{e_1, \ldots, e_n\}$ is  
\emph{any} natural basis, then \cite[Lemma 2.7]{EL} shows that 
$$
\ann(\cE) =\espan{e_i \in B : e_i^2 = 0}.
$$
\noindent The same argument applied to $ \cE / \ann(\cE)$ shows that 
$$
\ann^{2}(\cE)= \espan{e_i \in B : e_i^2 \in \ann(\cE)},
$$
\noindent and in general, for any $i$, 
$$
\ann^{i}(\cE)= \espan{e_i \in B : e_i^2 \in \ann^{i-1}(\cE)}.
$$
Then $B$ splits as the disjoint union
$$ B = B_1 \cup \cdots \cup B_r$$
where $B_i =  \{e \in B \; | \; e^2 \in \ann^{i-1}(\cE), \; e \notin \ann^{i-1}(\cE) \}$
 
 Then for all $i = 1,\ldots, r$, $B_1 \cup \cdots \cup B_r $ is a basis of $\ann^i(\cE)$. 
 In particular each $\ann^i(\cE)$ is an evolution ideal (that is, it is an ideal in the usual sense, and it
 is an evolution algebra too), and it can be easily computed from any natural basis.
 
 Let $\cU_i \bydef \espan{B_i}$ for all $i = 1,\ldots, r$, so 
 $\cU_1 \oplus  \cdots \oplus \cU_i = \ann^i(\cE)$, for all $i = 1,\ldots, r$. 

\begin{proposition}\label{pr:UiU1}
Let $\cE$ be a finite-dimensional nilpotent evolution algebra. Then
for all $i = 2,\ldots, r$, $\cU_i\oplus \cU_1 = \ann_{\ann^i(\cE)}(\ann^{i-1}(\cE))$, 
(i.e. $\cU_i \oplus \cU_1 = \{ x \in  \ann^i(\cE) \; | \; x\ann^{i-1}(\cE) = 0\}$).
 \end{proposition}
 \begin{proof}
 By definition of a natural basis $\cU_i(\cU_1 \oplus  \cdots \oplus \cU_{i-1}) = \cU_i \ann^{i-1}(\cE) = 0$ and 
 $\cU_i\subseteq \ann_{\ann^i(\cE)}(\ann^{i-1}(\cE))$. Hence
 $\ann_{\ann^i(\cE)}(\ann^{i-1}(\cE)) = \cU_i \oplus \ann_{\ann^{i-1}(\cE)}(\ann^{i-1}(\cE)) = \cU_i\oplus \cU_1$ 
 where the last equality follows from the result in \cite[Lemma 2.7]{EL}
 mentioned above.
  \end{proof}
  
  Therefore, the subspaces $\cU_i\oplus \cU_1$ are invariants of $\cE$ and do not depend on the natural basis chosen.
  
\smallskip
  
Recall that, in general, there is no uniqueness of natural bases (see \cite{EL}), but for nilpotent evolution algebras, Proposition \ref{pr:UiU1} gives certain rigidity:
  
\begin{corollary}\label{cor:basechange}
Let $B^1$ and $B^2$ be two natural bases of a finite-dimensional nilpotent evolution algebra $\cE$, ordered so that the first elements are 
   in $\ann(\cE)$, the next ones in $\ann^2(\cE)\setminus\ann(\cE)$, ... Then 
   the matrix of the base change has the following block structure:
\[\begin{pmatrix}
*&*&* &\dots&*\\
0&*&0 &\dots&0\\
0&0&* &\dots&0\\
\vdots&\vdots&\vdots&\ddots&\vdots\\
0&0& 0&\dots&*&
\end{pmatrix}.\]
 
   \end{corollary}
  
 \bigskip

\section{Some families of nilpotent evolution algebras}\label{se: somefamilies-nilpotent_evol}

Let $\FF$ be a field of characteristic not $2$ and let $\cU$ be a nonzero finite-dimensional vector space over $\FF$ with $\dim_\FF\cU=n$.
We will define in this section some families of nilpotent evolution algebras of very specific types. These will be instrumental
in the classifications of low-dimensional nilpotent evolution algebras.

\smallskip

\begin{df}\label{df:Ub}
 Let  $b: \cU \times \cU \longrightarrow \FF$ be a nondegenerate, symmetric, bilinear form. We define the algebra
 $\cE(\cU,b) \bydef \cU \times \FF$ with multiplication 
 \[ 
(u,\alpha) (v,\beta) = (0, b(u,v)),
\]
for any $u,v\in\cU$ and $\alpha,\beta\in\FF$.
\end{df}

\begin{proposition}
$\cE(\cU,b) $ is a nilpotent evolution algebra of type $[1,n]$.
\end{proposition}
 \begin{proof}
 Let $\{u_1, \ldots, u_n\}$ be an orthogonal basis of $\cU$ relative to $b$. Then  $\{(u_1,0), \ldots, (u_n,0), (0,1)\}$ is a natural basis of $\cE =\cE(\cU,b)$ and
 $\ann(\cE)= 0 \times \FF$, $\ann^{2}(\cE)= \cE$.
   \end{proof}
   
\smallskip
\begin{df}\label{df:Ubf}
Let  $b: \cU \times \cU \longrightarrow \FF$ be a nondegenerate, symmetric, bilinear form and let $g:\cU \longrightarrow \cU$ be a symmetric endomorphism relative to $b$. 
Assume that $g$ is diagonalizable (this is always the case if $\FF=\RR$ and $b$ is a definite form). 
We define the algebra
 $\cE(\cU,b,g) \bydef \cU \times \FF \times \FF$ with multiplication 
\[
(u,\alpha, \beta ) (v,\gamma,\delta) = (0, b(u,v),b(g(u),v) + \alpha \gamma),
\]
for any $u,v\in\cU$ and $\alpha,\beta,\gamma,\delta\in\FF$.
\end{df}

\begin{proposition}
$\cE(\cU,b,g) $ is a nilpotent evolution algebra of type $[1,1,n]$.
\end{proposition}
 \begin{proof}
 Our assumptions imply that there is an orthogonal basis, relative to $b$, 
 consisting of eigenvectors of $g$: $\{u_1, \ldots, u_n\}$. 

Then 
 $\{(u_1,0,0), \ldots, (u_n,0,0), (0,1,0),(0,0,1)\}$ is a natural basis of $\cE =\cE(\cU,b,g)$. Besides
 $\ann(\cE)= \FF(0,0,1) = 0 \times 0 \times \FF$, $\ann^{2}(\cE)= 0 \times \FF \times \FF$, and 
 $\ann^{3}(\cE)=\cE$.
   \end{proof}

\smallskip

\begin{df}\label{df:Ubfg}
Let  $b: \cU \times \cU \longrightarrow \FF$ be a nondegenerate, symmetric, bilinear form and let
$f,g:\cU \longrightarrow \cU$ be two commuting, symmetric (relative to $b$), diagonalizable endomorphisms.  We define the algebra
 $\cE(\cU,b,f,g) \bydef \cU \times \FF\times \FF \times \FF$ with multiplication 
 \[
(u,\alpha, \beta, \gamma) (v,\epsilon,\delta,\eta ) = 
(0, b(u,v),b(f(u),v) + \alpha\epsilon , b(g(u),v) + \beta \delta),
\]
for any $u,v\in\cU$ and $\alpha,\beta,\gamma,\epsilon,\delta,\eta\in\FF$.
\end{df}

\begin{proposition}
 $\cE(\cU,b,f,g)$ is a nilpotent evolution algebra whose type is $[1,1,1,n]$.
\end{proposition}
 \begin{proof}
 Our assumptions imply that there is an orthogonal basis of $\cU$, relative to $b$, consisting of common eigenvectors for $f$ and $g$: $\{u_1, \ldots, u_n\}$. 

Then 
 $\{(u_1,0,0,0), \ldots, (u_n,0,0,0), (0,1,0,0),(0,0,1,0),(0,0,0,1)\}$ is a natural basis of $\cE =\cE(\cU,b,f,g)$. Besides,
 $\ann(\cE)= 0 \times 0 \times 0 \times \FF$, $\ann^{2}(\cE)= 0 \times 0 \times \FF \times \FF$,
 $\ann^{3}(\cE)= 0 \times \FF \times \FF \times \FF$, and 
 $\ann^{4}(\cE)=\cE$.
   \end{proof}
   
These algebras give all the nilpotent evolution algebras of types $[1,n]$, $[1,1,n]$ and $[1,1,1,n]$, up to isomorphism. The situation for type $[1,1,1,1,n]$ is more complicated, as shown by the classification in Section \ref{se:classificationdimfive}.
   
\begin{theorem}\label{teo:threetypes}
Let $\cE$ be  nilpotent evolution algebra.
\begin{romanenumerate}
\item 
If $\cE$ is of type $[1,n]$, then $\cE$ is isomorphic to $\cE(\cU,b)$ for some $(\cU,b)$ as in 
Definition \ref{df:Ub}. Moreover, for two such pairs $(\cU,b)$ and $(\cU',b')$, $\cE(\cU,b)$ is 
isomorphic to $\cE(\cU',b')$ if and only if there is a similarity $\psi:(\cU,b)\rightarrow (\cU',b')$, 
i.e., a linear isomorphism $\psi:\cU\rightarrow\cU'$ such that there is a nonzero scalar $\mu$, called the norm of 
$\psi$, such that $b'(\psi(u), \psi(v)) = \mu b(u,v)$ for all $u,v \in \cU$. 

\item  If $\cE$ is of type $[1,1,n]$, then $\cE$ is isomorphic to $\cE(\cU,b,g)$ for some $(\cU,b,g)$ as in Definition 
\ref{df:Ubf}. Moreover, for two such triples $(\cU,b,g)$ and $(\cU',b',g')$, 
$\cE(\cU,b,g)$ is isomorphic to $\cE(\cU',b',g')$ if and only if 
there is a similarity $\psi: (\cU,b) \longrightarrow (\cU',b')$ 
and a scalar $\nu \in \FF$ such that $ \psi^{-1} \circ g' \circ \psi  = \mu g + \mu^{-1} \nu \id$, 
where $\mu $ is the norm of $\psi$. 

\item 
If $\cE$ is of type $[1,1,1,n]$, then $\cE$ is isomorphic to  $\cE(\cU,b,f,g)$ for some $(\cU,b,f,g)$ as in Definition 
\ref{df:Ubfg}. Moreover, for two such quadruples $(\cU,b,f,g)$ and $(\cU',b',f',g')$, the corresponding algebras
$\cE(\cU,b,f,g)$ and $\cE(\cU',b',f',g')$ are isomorphic if and only if there is a 
similarity $\psi: (\cU,b) \longrightarrow (\cU',b')$ and a scalar $\nu \in \FF$ such 
that $\psi^{-1} \circ f' \circ \psi  = \mu f$ and $\psi^{-1} \circ g' \circ \psi  = \mu^3 g + \mu^{-1} \nu \id$, 
where $\mu $ is the norm of $\psi$.
\end{romanenumerate}
\end{theorem}
 \begin{proof}
 We will give the proof of (iii), because (i) and (ii) are simpler.  Let $\cE$ be of type $[1,1,1,n]$, 
 $ \{u_1, \ldots, u_n, w, t ,s\}$ a natural basis with
 $\ann(\cE)= \FF s$, $\ann^{2}(\cE) = \FF t \oplus \FF s$, and $\ann^{3}(\cE) = \FF w \oplus \FF t \oplus \FF s$. 
 Then there are scalars $\alpha, \beta, \gamma  \in \FF$, $\alpha \neq 0 \neq \gamma $, 
 such that $w^2 = \alpha t + \beta s$, $t^2 = \gamma s$. 
 But $\{u_1, \ldots, u_n, w, w^2 , (w^2)^2\}$ is another natural basis so we may assume, without 
 loss of generality, that $w^2 = t$ and $t^2 = s$. For all $i =1, \ldots,n$, 
 $u_i^2 = \lambda_iw + \mu_i t +\nu_i s$ with $ \lambda_i, \mu_i, \nu_i \in \FF$, $\lambda_i \neq 0$ because 
 $u_i\notin \ann^3(\cE)$. 
 
 Let $\cU = \FF u_1 \oplus  \cdots \oplus \FF u_n $ and define $b: \cU \times \cU \longrightarrow \FF$  by $b(u_i, u_j) = 0 $ for $ i \neq j$,
 $b(u_i, u_i) =\lambda_i$ for any $i=1,\ldots,n$. Define $f$ (respectively $g$) by $f(u_i) = \lambda_i^{-1} \mu_i u_i$  (respectively $g(u_i) = \lambda_i^{-1} \nu _i u_i$).
 
 Then the linear map  $\varphi: \cE \longrightarrow \cE(\cU,b,f,g) $ such that 
\[ 
u \mapsto (u,0,0,0),\ 
w \mapsto (0, 1,0,0),\ 
t \mapsto (0,0,1,0),\ 
s \mapsto (0,0,0,1),
\]
is  an isomorphism of algebras.
 
For a subalgebra $\cS$ of an algebra $\cA$, $\ann(\cS)$ denotes the subspace $\{x\in\cA\;|\; x\cS=\cS x=0\}$.
 If $\phi:  \cE(\cU,b,f,g) \longrightarrow \cE(\cU',b',f',g')$ is an isomorphism of algebras, then 
 since $\ann(\ann^{3}(\cE(\cU,b,f,g)) )= \ann(0\times \FF \times \FF \times \FF) = \cU\times 0 \times 0 \times \FF$, and similarly for $\cE(\cU',b',f',g')$, 
 for any $u \in \cU$, $\phi((u,0,0,0)) = (\psi(u),0,0,\chi(u))$ for a linear isomorphism 
 $\psi: \cU \longrightarrow \cU'$, and a linear form $\chi: \cU \longrightarrow \FF$. 

Also $\ann_{\ann^{3}(\cE(\cU,b,f,g))}(\ann^{2}(\cE(\cU,b,f,g))) = 0 \times \FF \times 0 \times \FF$, 
and similarly for $\cE(\cU',b',f',g')$, 
so $\phi((0,1,0,0)) = (0, \mu, 0, \nu)$ for some $\mu, \nu \in \FF$, $\mu \neq 0$.
 Then 
 \[
  \begin{split}
 \phi((0,0,1,0))  &= \phi((0,1,0,0)^2) =(0, \mu, 0, \nu)^2 = (0, 0, \mu^2, 0),\\
 \phi((0,0,0,1))  &= \phi((0,0,1,0)^2) =(0, 0, \mu^2, 0)^2 = (0, 0, 0, \mu^4).
 \end{split}
 \]
 
Now, for $ u,v \in \cU$, 
\[
\begin{split}
&\phi((u,0,0,0) (v,0,0,0))= \phi((0, b(u,v), b(f(u),v), b(g(u),v))\\
    &\qquad\qquad= \bigl(0, \mu b(u,v), \mu^2 b(f(u),v), \nu b(u,v) +\mu^4 b(g(u),v)\bigr),\\[2pt]
&\phi((u,0,0,0)) \phi ((v,0,0,0)) = (\psi(u),0,0,\chi(u))(\psi(v),0,0,\chi(v)) \\
 &\qquad\qquad= \bigl(0, b'(\psi(u),\psi(v)), b'(f'(\psi(u)), \psi(v)), b'(g'(\psi(u)), \psi(v)\bigr).
 \end{split}
 \]
 Therefore,
 \begin{itemize}
  \item 
  $b'(\psi(u),\psi(v)) = \mu b(u,v)$, for all $ u,v \in \cU$, i.e. $\psi $ is a similarity of norm $\mu$.
  
   \item 
   $b'(f'(\psi(u)), \psi(v)) = \mu^2 b(f(u),v)=\mu b'(\psi(f(u)),\psi(v))$ for all $u,v \in \cU$, so $\psi^{-1}\circ f' \circ \psi = \mu f$.
   
    \item 
    $b'(g'(\psi(u)), \psi(v)) =\nu b(u,v) + \mu^4 b(g(u),v)$  for all $u,v \in \cU$, 
    so $\psi^{-1}\circ g' \circ \psi = \mu^3 g + \mu^{-1} \nu \id$.
 \end{itemize}
 
 Conversely, if  $\psi: (\cU,b) \longrightarrow (\cU',b')$  is a similarity of norm  
 $\mu$ such that $\psi^{-1}\circ f' \circ \psi = \mu f$, and
 $\psi^{-1}\circ g' \circ \psi = \mu^3 g + \mu^{-1} \nu \id$ for a scalar $\nu$, then the linear isomorphism  
 $\phi: \cE(\cU,b,f,g) \longrightarrow \cE(\cU',b',f',g') $ such that 
 $(u, \alpha, \beta, \gamma ) \mapsto (\psi(u), \mu \alpha, \mu^2 \beta, \mu^4 \gamma + \nu \alpha )$ is an 
 isomorphism of algebras.
   \end{proof}
   
\begin{corollary} \label{co:Ubfg}
    $\cE(\cU,b,f,g)$ is isomorphic to $\cE(\cU, \alpha b,\alpha f, \alpha^3 g + \beta \id)$ 
    for any $\alpha, \beta  \in \FF, \alpha \neq 0$.
\end{corollary}
   
\smallskip
   
 \begin{df}\label{df:1n1}
 Let  $b: \cU \times \cU \longrightarrow \FF$ be a nondegenerate, symmetric, bilinear form and $0 \neq u \in \cU$. We define the algebra
 $\cE(\cU,b, u) \bydef \FF \times \cU \times \FF$ with multiplication 
 \[
(\alpha,x,\beta) (\gamma,y,\delta) = (0,\alpha \gamma u,  b(x,y)),
\]
for any $x,y\in\cU$ and $\alpha,\beta,\gamma,\delta\in\FF$.
\end{df}

\begin{proposition}
$\cE(\cU,b,u) $ is a nilpotent evolution algebra of type $[1,n,1]$.
\end{proposition}
 \begin{proof}
 Let $\{u_1, \ldots, u_n\}$ be an orthogonal basis of $\cU$ relative to $b$. Then  $\{(1,0,0), (0,u_1,0), \ldots, (0,u_n,0), (0,0,1)\}$ is a natural basis of $\cE =\cE(\cU,b,u)$. Besides
 $\ann(\cE)= 0 \times 0 \times \FF$, $\ann^{2}(\cE) = 0 \times \cU \times \FF$, and $\ann^{3}(\cE)= \cE$.
   \end{proof}
   
\begin{theorem}\label{teo:onetype}
Let $\cE$ be  nilpotent evolution algebra of type $[1,n,1]$. Then $\cE$ is isomorphic to an algebra
$\cE(\cU,b,u)$, for some $(\cU,b,u)$ as in Definition \ref{df:1n1}. Moreover, for any two such triples $(\cU,b,u)$ and 
$(\cU',b',u')$, the algebras $\cE(\cU,b,u)$ and $\cE(\cU',b',u')$ are isomorphic if and only if there is 
a similarity $\psi: (\cU,b) \longrightarrow (\cU',b')$ such that $\psi(u) = u'$.
\end{theorem}
 \begin{proof}
 Let $\cE$ be of type $[1,n,1]$, $B = \{a,u_1, \ldots, u_n,s\}$  a natural basis of $\cE$ with  
 $\ann(\cE)= \FF s$, 
 $\ann^{2}(\cE) = \FF u_1 \oplus  \cdots \oplus \FF u_n \oplus \FF s$, and  $\ann^{3}(\cE) = \cE$. 
 Then, $s^2 = 0$, $u_i^2 = \lambda_i s$, with $0\ne\lambda_i\in\FF$, for all $i=1,\ldots,n$, and
 $a^2 = \alpha_1 u_1 + \cdots + \alpha_n u_n + \alpha s$ for $\alpha_1,
\ldots, \alpha_n,\alpha\in\FF$,  and $\alpha_i \neq 0$ for at least one $ i = 1, \ldots, n$. Then $\{a,u_1, \ldots, u_{i-1}, u_i + \alpha_i^{-1}\alpha s, u_{i+1}, \ldots, u_n, s\}$ is another natural basis. Hence we may assume that $a^2 = \alpha_1 u_1 + \cdots + \alpha_n u_n $.

Let $\cU = \FF u_1 \oplus  \cdots \oplus \FF u_n$, 
$u = \alpha_1 u_1 + \cdots + \alpha_n u_n$, and $b: \cU \times \cU \longrightarrow \FF$ the nondegenerate, 
symmetric, bilinear form given by $b(u_i, u_j) = 0 $ for $ i \neq j$, $b(u_i, u_i) =\lambda_i$ for  $ i = 1, \ldots, n$.

The linear map $\varphi: \cE \longrightarrow \cE(\cU,b,u) $ such that 
$a \mapsto (1,0,0), \; u_i \mapsto (0, u_i,0), \; i = 1, \ldots, n, \; s \mapsto (0,0,1)$ is then 
an isomorphism of algebras. 

If $\phi:  \cE(\cU,b,u) \longrightarrow \cE(\cU',b',u')$ is an isomorphism of algebras, then since 
$\ann(\ann^{2}(\cE(\cU,b,u))) = \FF \times 0 \times \FF$, it follows that 
$\phi((1,0,0)) = (\mu,0,\nu )$ for some $\mu,\nu  \in \FF, \mu \neq 0$. Also $\phi$ restricts to an isomorphism 
$ \ann^{2}(\cE(\cU,b,u))\longrightarrow \ann^{2}(\cE(\cU',b',u'))$,
and these latter algebras are of type $[1,n]$. 
Hence there exists a similarity $\psi: (\cU,b) \longrightarrow (\cU',b')$ and a linear form 
$\chi: \cU \longrightarrow \FF$ such that 
$\phi(0,x,0) = (0, \psi(x), \chi(x))$ for all $x \in \cU$. Then 
\[
\phi((0,u,0))= \phi((1,0,0)^2) = \phi((1,0,0))^2 = (\mu, 0, \nu)^2 = (0, \mu^2u',0),
\]
so $\psi(u) = \mu^2 u'$. The similarity $\bar{\psi} = \mu^{-2}\psi $ satisfies $\bar{\psi}(u) = u'$. 

Conversely, if $\psi: (\cU,b) \longrightarrow (\cU',b')$ is a similarity with $\psi(u) =u'$, then 
the linear map $\phi:  \cE(\cU,b,u) \longrightarrow \cE(\cU',b',u')$, 
$(\alpha,x,\beta) \mapsto (\alpha, \psi(x), \mu \beta)$, where $\mu$ is the norm of $\psi$, is 
an isomorphism of algebras.
\end{proof}

\begin{corollary}\label{cor:type[1,n1]}
If $\FF$ is algebraically closed and $n \geq 2$, there are, up to isomorphism, exactly 
two nilpotent evolution algebras of type $[1,n,1], $ namely 
$$ 
\cE\bigl(\FF^n, b_n, (1,0, \cdots,0)\bigr)  \ \text{and}\;  \  \cE\bigl(\FF^n, b_n, (1,i,0 \ldots,0)\bigr)
$$
where $b_n((\alpha_1, \ldots,\alpha_n),(\beta_1, \ldots,\beta_n))
=\alpha_1\beta_1 +\cdots +\alpha_n \beta_n$, and $i$ is a fixed square root of $-1$.
 \end{corollary}
 \begin{proof}
 Any pair $(\cU,b)$ where 
$\dim_{\FF}(\cU) = n $ and $b$ is a nondegenerate, symmetric, bilinear form on $\cU$, is isometric to
$(\FF^n, b_n)$.
 Now, any nonisotropic vector in $\FF^n$ lies in the orbit of $(1,0, \ldots,0)$ under the group of similarities 
 of $(\FF^n, b_n)$, while any nonzero isotropic vector lies in the orbit of 
 $(1,i,0 \ldots,0)$. The result then follows from the previous Theorem.
   \end{proof}

\bigskip
\section{Classification of nilpotent evolution algebras up to dimension four}\label{se:classificationdimfour}


By Theorem \ref{teo:Krull-Schmidt Theorem} it is enough to classify indecomposable algebras. Up  
to dimension four, everything follows from the previous results easily.  The algebras in the classification 
list will be given by describing their weighted  graphs $\Gamma^w(\cE, B) $ in a suitable natural basis. 
If no weight is assigned to an
edge of this graph, it will be understood that the weight is $1$.

   \begin{theorem}\label{teo:indecomposabledimfour}
   Let $\cE$ be an indecomposable nilpotent evolution algebra, of dimension at most four,  
   over an algebraically closed field of characteristic not two. Then $\cE$ is isomorphic to one and only one of the algebras whose graphs are in Tables \ref{tb:dim3} and \ref{tb:dim4}.
\end{theorem}

\begin{table}[!ht]
\begin{center}
 \begin{tabular}{|c|c|c|}
  \hline
   \null\ $\dim\cE$\ \null&\null\  Type of $\cE$ \ \null
&\null\qquad Graph\qquad\null\\
\hline
\hline & & \\[-8pt]  
1 & [1] & \begin{minipage}[m]{70pt}{\[
\begin{tikzpicture}[x=1.5pt, y=1.5pt]

  \tikzset{vertex/.style = {shape=circle,draw, inner sep=0pt, minimum size=3pt}}
  \tikzset{edge/.style = {->,> = latex'}}
  \node[vertex] (a) at  (0,0) {};
  
  \end{tikzpicture}
   \]\vspace{-8pt}}
 \end{minipage}
  \\
\hline  &&\\[-8pt]  
2 & [1,1] & \begin{minipage}[m]{70pt}{\[
\begin{tikzpicture}[x=2pt, y=2pt]

 \tikzset{vertex/.style = {shape=circle,draw, inner sep=0pt, minimum size=3pt}}
 \tikzset{edge/.style = {->,> = latex'}}
 \node[vertex] (a) at  (0,10) {};
 \node[vertex] (b) at  (0,0) {};

 \draw[edge] (a) to (b);

 \end{tikzpicture}
 \]\vspace{-8pt}}
 \end{minipage}
  \\ 
  \hline  && \\[-8pt]  
3 & [1,2] & \begin{minipage}[m]{70pt}{\[
\begin{tikzpicture}[x=1.5pt, y=1.5pt]

 \tikzset{vertex/.style = {shape=circle,draw, inner sep=0pt, minimum size=3pt}}
 \tikzset{edge/.style = {->,> = latex'}}
 \node[vertex] (a) at  (-6,10) {};
 \node[vertex] (b) at  (6,10) {};
 \node[vertex] (c) at  (0,0) {};

 \draw[edge] (a) to (c);
 \draw[edge] (b) to (c);

 \end{tikzpicture}
\]\vspace{-8pt}}
 \end{minipage}
  \\ 
 \hline  && \\[-8pt]  
3 & [1,1,1] & \begin{minipage}[m]{70pt}{\[
 \begin{tikzpicture}[x=1.5pt, y=1.5pt]

 \tikzset{vertex/.style = {shape=circle,draw, inner sep=0pt, minimum size=3pt}}
 \tikzset{edge/.style = {->,> = latex'}}
 \node[vertex] (a) at  (0,20) {};
 \node[vertex] (b) at  (0,10) {};
 \node[vertex] (c) at  (0,0) {};
 
 \draw[edge] (a) to (b);
 \draw[edge] (b) to (c);

 \end{tikzpicture}
\]\vspace{-8pt}}
 \end{minipage}
  \\
 \hline
 \end{tabular}
\end{center}
\caption{\vrule width 0pt height 12pt$\dim\cE\leq 3$}
\label{tb:dim3}
\end{table}
 
 \begin{table}[!ht]
\begin{center}
 \begin{tabular}{|c|c|c|c|}
  \hline
   \null\ $\dim\cE$\ \null&\null\  Type of $\cE$ \ \null
&\null\qquad Graph\qquad\null&\null\qquad Properties\qquad\null\\
\hline
 \hline  &&& \\[-8pt]  
4 & [1,3] & \begin{minipage}[m]{70pt}{\[
\begin{tikzpicture}[x=1.5pt, y=1.5pt]

 \tikzset{vertex/.style = {shape=circle,draw, inner sep=0pt, minimum size=3pt}}
 \tikzset{edge/.style = {->,> = latex'}}
 \node[vertex] (a) at  (-10,10) {};
 \node[vertex] (b) at  (0,10) {};
 \node[vertex] (c) at  (10,10) {};
 \node[vertex] (d) at  (0,0) {};

 \draw[edge] (a) to (d);
 \draw[edge] (b) to (d);
 \draw[edge] (c) to (d);

 \end{tikzpicture}
\]\vspace{-8pt}}
 \end{minipage}
 & \\
  \hline  &&& \\[-8pt]  
4 & [1,2,1] & \begin{minipage}[m]{70pt}{\[
\begin{tikzpicture}[x=1.5pt, y=1.5pt]

 \tikzset{vertex/.style = {shape=circle,draw, inner sep=0pt, minimum size=3pt}}
 \tikzset{edge/.style = {->,> = latex'}}
 \node[vertex] (a) at  (-10,20) {};
 \node[vertex] (b) at  (-10,10) {};
 \node[vertex] (c) at  (10,10) {};
 \node[vertex] (d) at  (0,0) {};

 \draw[edge] (a) to (b);
 \draw[edge] (b) to (d);
 \draw[edge] (c) to (d);

 \end{tikzpicture}
\]\vspace{-8pt}}
 \end{minipage}
 &$((\cU_3\oplus\cU_1)^2)^2\ne 0$ \\ 
 \hline  &&& \\[-8pt]  
4 & [1,2,1] & \begin{minipage}[m]{70pt}{\[
\begin{tikzpicture}[x=1.5pt, y=1.5pt]

\tikzset{vertex/.style = {shape=circle,draw, inner sep=0pt, minimum size=3pt}}
\tikzset{edge/.style = {->,> = latex'}}
\node[vertex] (a) at  (0,20) {};
\node[vertex] (b) at  (-10,10) {};
\node[vertex] (c) at  (10,10) {};
\node[vertex] (d) at  (0,0) {};

\draw[edge] (a) to (b);
\draw[edge] (a) to node[near start,right]{\tiny$i$} (c);
\draw[edge] (b) to (d);
\draw[edge] (c) to (d);

\end{tikzpicture}
 \]\vspace{-8pt}}
 \end{minipage}
 &$((\cU_3\oplus\cU_1)^2)^2 = 0$ \\ 
  \hline  &&& \\[-8pt]  
4 & [1,1,2] & \begin{minipage}[m]{70pt}{\[
\begin{tikzpicture}[x=1.5pt, y=1.5pt]

\tikzset{vertex/.style = {shape=circle,draw, inner sep=0pt, minimum size=3pt}}
\tikzset{edge/.style = {->,> = latex'}}
\node[vertex] (a) at  (-6,20) {};
\node[vertex] (b) at  (6,20) {};
\node[vertex] (c) at  (0,10) {};
\node[vertex] (d) at  (0,0) {};

\draw[edge] (a) to (c);
\draw[edge] (b) to (c);
\draw[edge] (c) to (d);

\end{tikzpicture}
\]\vspace{-8pt}}
 \end{minipage}
 &$\dim(\cU_3\oplus\cU_1)^2 = 1$ \\ 
\hline  &&& \\[-8pt]  
4 & [1,1,2] & \begin{minipage}[m]{70pt}{\[
\begin{tikzpicture}[x=1.5pt, y=1.5pt]

\tikzset{vertex/.style = {shape=circle,draw, inner sep=0pt, minimum size=3pt}}
\tikzset{edge/.style = {->,> = latex'}}
\node[vertex] (a) at  (-6,20) {};
\node[vertex] (b) at  (6,20) {};
\node[vertex] (c) at  (0,10) {};
\node[vertex] (d) at  (0,0) {};

\draw[edge] (a) to (c);
\draw[edge] (b) to (c);
\draw[edge] (c) to (d);
\draw[edge] (b) to[bend left] (d);

\end{tikzpicture}
\]\vspace{-8pt}}
 \end{minipage}
 &$\dim(\cU_3\oplus\cU_1)^2 = 2$ \\ 
 \hline  &&& \\[-8pt]  
4 & [1,1,1,1] & \begin{minipage}[m]{70pt}{\[
\begin{tikzpicture}[x=1.5pt, y=1.5pt]

\tikzset{vertex/.style = {shape=circle,draw, inner sep=0pt, minimum size=3pt}}
\tikzset{edge/.style = {->,> = latex'}}
\node[vertex] (a) at  (0,30) {};
\node[vertex] (b) at  (0,20) {};
\node[vertex] (c) at  (0,10) {};
\node[vertex] (d) at  (0,0) {};

\draw[edge] (a) to (b);
\draw[edge] (b) to (c);
\draw[edge] (c) to (d);

\end{tikzpicture}
 \]\vspace{-8pt}}
 \end{minipage}
 & $(\cU_4\oplus\cU_1)^2 \subseteq\cU_3\oplus\cU_1 $\\ 
 \hline  &&& \\[-8pt]  
4 & [1,1,1,1] & \begin{minipage}[m]{70pt}{\[
 \begin{tikzpicture}[x=1.5pt, y=1.5pt]

\tikzset{vertex/.style = {shape=circle,draw, inner sep=0pt, minimum size=3pt}}
\tikzset{edge/.style = {->,> = latex'}}
\node[vertex] (a) at  (0,30) {};
\node[vertex] (b) at  (0,20) {};
\node[vertex] (c) at  (0,10) {};
\node[vertex] (d) at  (0,0) {};

\draw[edge] (a) to (b);
\draw[edge] (a) to[bend left=40] (c);
\draw[edge] (b) to (c);
\draw[edge] (c) to (d);

\end{tikzpicture}
 \]\vspace{-8pt}}
 \end{minipage}
 & $(\cU_4\oplus\cU_1)^2 \nsubseteq\cU_3\oplus\cU_1 $\\ 
   \hline
 \end{tabular}
\end{center}
\caption{\vrule width 0pt height 12pt$\dim\cE= 4$}
\label{tb:dim4}
\end{table}


 \begin{proof}
 If $\dim(\cE) = 1$ this is trivial. For dimension $2$ or $3$, by Corollary \ref{cor:dimensionann}, 
 we have $\dim(\ann(\cE)) = 1$  (otherwise $\cE$ would be decomposable), so the possible types are 
$[1,1], [1,2]$ or $[1,1,1]$, and the result follows from Theorem \ref{teo:threetypes}. 
If  $\dim(\cE)= 4$, then $\dim(\ann(\cE))= 1$ by Corollary 
\ref{cor:dimensionann}. The possible types are (ordered lexicographically) 
$[1,3], [1,2,1], [1,1,2]$ and $[1,1,1,1]$, and the result follows from Theorems \ref{teo:threetypes} and  \ref{teo:onetype}.
   \end{proof}

   \begin{remark}
   
The algebras with graphs
\[
\begin{tikzpicture}[x=1.5pt, y=1.5pt]

\tikzset{vertex/.style = {shape=circle,draw, inner sep=0pt, minimum size=3pt}}
\tikzset{edge/.style = {->,> = latex'}}
\node[vertex] (a) at  (0,20) {};
\node[vertex] (b) at  (-10,10) {};
\node[vertex] (c) at  (10,10) {};
\node[vertex] (d) at  (0,0) {};

\draw[edge] (a) to (b);
\draw[edge] (a) to node[near start,right]{\tiny$i$} (c);
\draw[edge] (b) to (d);
\draw[edge] (c) to (d);

\end{tikzpicture}
\quad\raisebox{15pt}{\text{and}}\quad
 \begin{tikzpicture}[x=1.5pt, y=1.5pt]

\tikzset{vertex/.style = {shape=circle,draw, inner sep=0pt, minimum size=3pt}}
\tikzset{edge/.style = {->,> = latex'}}
\node[vertex] (a) at  (0,30) {};
\node[vertex] (b) at  (0,20) {};
\node[vertex] (c) at  (0,10) {};
\node[vertex] (d) at  (0,0) {};

\draw[edge] (a) to (b);
\draw[edge] (a) to[bend left=40] (c);
\draw[edge] (b) to (c);
\draw[edge] (c) to (d);

\end{tikzpicture}
\]
are missing in   \cite{HA1}.
   
   \end{remark}
 

\bigskip
  
   
\section{Classification of five-dimensional nilpotent evolution algebras}
\label{se:classificationdimfive}

First we will classify the indecomposable, five-dimensional, nilpotent, evolution algebras $\cE$ with $\dim(\ann(\cE)) > 1$.

 \begin{theorem}
   Let $\cE$ be an indecomposable, five-dimensional,  nilpotent, evolution algebra,  over an algebraically closed field of characteristic not two. Then:
   \begin{romanenumerate}
    \item The dimension of $\ann(\cE) $ is $1$ or $2$.
    \item If the dimension of $\ann(\cE) $ is $2$, then $\cE$ is isomorphic to one and only one of the algebras with graph in Table \ref{tb:ann2}.
     \end{romanenumerate}
\end{theorem}
  

\begin{table}[!ht]
\begin{center}
 \begin{tabular}{|c|c|}
  \hline
  \null\qquad Type of $\cE$ \qquad\null&\null\qquad Graph\qquad\null\\
\hline
\hline &\\[-12pt]
  [2,3]& \begin{minipage}[m]{70pt}{
  \[
 \begin{tikzpicture}[x=1.5pt, y=1.5pt]
 \tikzset{vertex/.style = {shape=circle,draw, inner sep=0pt, minimum size=3pt}}
 \tikzset{edge/.style = {->,> = latex'}}
 \node[vertex] (a) at  (-10,10) {};
 \node[vertex] (b) at  (10,10) {};
 \node[vertex] (c) at  (0,0) {};
\node[vertex] (d) at  (20,0) {};
\node[vertex] (e) at  (30,10) {};

 \draw[edge] (a) to (c);
 \draw[edge] (b) to (c);
  \draw[edge] (b) to (d);
   \draw[edge] (e) to (d);

 \end{tikzpicture}
 \]\vspace{-8pt}}
 \end{minipage}\\
  \hline&\\[-12pt]
   [2,2,1]&\begin{minipage}[m]{70pt}{
   \[
 \begin{tikzpicture}[x=1.5pt, y=1.5pt]
 \tikzset{vertex/.style = {shape=circle,draw, inner sep=0pt, minimum size=3pt}}
 \tikzset{edge/.style = {->,> = latex'}}
 \node[vertex] (a) at  (10,10) {};
 \node[vertex] (b) at  (0,0) {};
\node[vertex] (c) at  (20,0) {};
\node[vertex] (d) at  (0,-10) {};
\node[vertex] (e) at  (20,-10) {};

 \draw[edge] (a) to (b);
  \draw[edge] (a) to (c);
  \draw[edge] (b) to (d);
  \draw[edge] (c) to (e);

 \end{tikzpicture}
 \]\vspace{-8pt}}
 \end{minipage}\\
   \hline&\\[-12pt]
  [2,1,2]&\begin{minipage}[m]{70pt}{
  \[
 \begin{tikzpicture}[x=1.5pt, y=1.5pt]
 \tikzset{vertex/.style = {shape=circle,draw, inner sep=0pt, minimum size=3pt}}
 \tikzset{edge/.style = {->,> = latex'}}
 \node[vertex] (a) at  (-10,10) {};
 \node[vertex] (b) at  (10,10) {};
 \node[vertex] (c) at  (0,0) {};
\node[vertex] (d) at  (0,-10) {};
\node[vertex] (e) at  (20,-10) {};
 \draw[edge] (a) to (c);
 \draw[edge] (b) to (c);
  \draw[edge] (c) to (d);
   \draw[edge] (b) to (e);

 \end{tikzpicture}
 \]\vspace{-8pt}}
 \end{minipage}\\
  \hline 
 \end{tabular}
\end{center}
\caption{\vrule width 0pt height 12pt$\dim\cE=5$, $\dim(\ann(\cE))=2$}
\label{tb:ann2}
\end{table}

 \begin{proof}
 The first assertion follows from Corollary \ref{cor:dimensionann}. Assume that $\cE$  is an indecomposable, 
 five-dimensional,  nilpotent, evolution algebra,  with two-dimensional annihilator. 
 The possible types of $\cE$, ordered lexicographically, are $[2,3]$, $[2,2,1]$, $[2,1,2]$, and $ [2,1,1,1]$.
 
 \begin{itemize}
 
 \medskip
 
  \item  If the type is $[2,3]$, let $\{x,y,z,u,v\}$ be a natural basis with $u^2 = v^2 = 0$ (that is, 
 $\ann(\cE) = \espan{u,v})$. Besides, by
 Corollary  \ref{cor:algebracuadann}, $\ann(\cE) = \cE^2  = \espan{x^2, y^2, z^2}$ and hence we may assume that $x^2$ and $z^2$ are linearly independent. Then 
 $\{x,y,z,u' = x^2,v' = z^2 \}$ is another   natural basis and $y^2 = \alpha u' + \beta  v'$ for 
 $\alpha, \beta  \in \FF$, $(\alpha, \beta) \neq (0,0)$ because
 $y \notin \ann(\cE)$.  If $\alpha  = 0$, then  $\cE  = \espan{x, u'} \oplus \espan{y,z, v'}$ 
 would be decomposable. The same happens if $\beta 
 = 0$. 
 Hence $\alpha \neq 0 \neq \beta $ and the graph in the natural basis $\{\sqrt{\alpha }x,y,\sqrt{\beta }z$, 
 $\alpha u',\beta v'\}$ is 
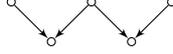
\[
 \begin{tikzpicture}[x=1.5pt, y=1.5pt]
 \tikzset{vertex/.style = {shape=circle,draw, inner sep=0pt, minimum size=3pt}}
 \tikzset{edge/.style = {->,> = latex'}}
 \node[vertex] (a) at  (-10,10) {};
 \node[vertex] (b) at  (10,10) {};
 \node[vertex] (c) at  (0,0) {};
\node[vertex] (d) at  (20,0) {};
\node[vertex] (e) at  (30,10) {};

 \draw[edge] (a) to (c);
 \draw[edge] (b) to (c);
  \draw[edge] (b) to (d);
   \draw[edge] (e) to (d);

 \end{tikzpicture}
\]
Notice that for any $p \neq q$ in a natural basis with   $p,q  \notin \ann(\cE)$, $p^2$ and $ q^2$ are linearly independent, so $\cE$ is not a direct sum of ideals with graphs
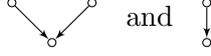
\[
 \begin{tikzpicture}[x=1.5pt, y=1.5pt]
 \tikzset{vertex/.style = {shape=circle,draw, inner sep=0pt, minimum size=3pt}}
 \tikzset{edge/.style = {->,> = latex'}}
 \node[vertex] (a) at  (-10,10) {};
 \node[vertex] (b) at  (10,10) {};
 \node[vertex] (c) at  (0,0) {};

 \draw[edge] (a) to (c);
 \draw[edge] (b) to (c);

 \end{tikzpicture}
 \quad\raisebox{8pt}{\text{and}}\quad
  \begin{tikzpicture}[x=1.5pt, y=1.5pt]

 \tikzset{vertex/.style = {shape=circle,draw, inner sep=0pt, minimum size=3pt}}
 \tikzset{edge/.style = {->,> = latex'}}
 \node[vertex] (a) at  (0,10) {};
 \node[vertex] (b) at  (0,0) {};

 \draw[edge] (a) to (b);

 \end{tikzpicture}
\]
It follows that $\cE$ is indeed indecomposable.
    
\medskip
    \item If the type is $[2,2,1]$, and $\{x,a,b,u,v\}$ is a natural basis with  $\ann(\cE)= 
    \espan{u,v}$ and $\ann^2(\cE)= \espan{a,b,u,v}$,
    then $x^2 = \alpha a + \beta b + p$, $a^2 = q$, $b^2 = r$, with $\alpha, \beta  \in \FF$, 
    $(\alpha, \beta) \neq (0,0)$, $p, q, r \in \ann(\cE)$.
    Since $\ann(\cE) $ is contained in $\cE^2$ because of Corollary \ref{cor:algebracuadann}, and 
    $\cE^2 =\espan{x^2,a^2, b^2}$, it follows that $a^2$ and $ b^2$ are linearly independent. Assume, without loss of generality, that $\alpha \neq 0$.
    
    If $\beta = 0$, then $\cE = \espan{x,   \alpha a + p, a^2} \oplus \espan{b, b^2}$ would be decomposable.
   Hence $\beta \neq 0$ and the graph in the natural basis $\{x, \alpha a + p, \beta b, \alpha^2 a^2, \beta^2 b^2 \}$ is
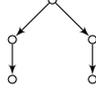
\[
 \begin{tikzpicture}[x=1.5pt, y=1.5pt]
 \tikzset{vertex/.style = {shape=circle,draw, inner sep=0pt, minimum size=3pt}}
 \tikzset{edge/.style = {->,> = latex'}}
 \node[vertex] (a) at  (10,10) {};
 \node[vertex] (b) at  (0,0) {};
\node[vertex] (c) at  (20,0) {};
\node[vertex] (d) at  (0,-10) {};
\node[vertex] (e) at  (20,-10) {};

 \draw[edge] (a) to (b);
  \draw[edge] (a) to (c);
  \draw[edge] (b) to (d);
  \draw[edge] (c) to (e);

 \end{tikzpicture}
 \]
Observe that this algebra is indecomposable  because $\dim\bigl(\ann(\cU_3+\cU_1)^2\bigr)=3$, while this dimension is $4$ for the only five-dimensional decomposable nilpotent evolution algebra with the same type and same $\dim\cE^2$ (which equals $3$). 

\medskip
        
        \item  If the type is $[2,1,2]$, there is  a natural basis $\{x,y,a,u,v\}$ with 
        $\ann(\cE) = \espan{u,v}$, $\ann^2(\cE) = \espan{a,u,v}$. 
        As $x^2 \in \ann^2(\cE)\setminus \ann(\cE)$, $\{x,y,x^2,u,v\}$ is another natural basis. 
        Also $y^2 \in \ann^2(\cE) \setminus \ann(\cE)$, so after scaling $y$ we may 
assume $y^2 = x^2 + u'$ with $u' \in \ann(\cE)$. Then $\cE^2= \espan{x^{2}, u', (x^{2})^2}$ and $\cE^2$ contains $\ann(\cE)$ by indecomposability
    (Corollary \ref{cor:algebracuadann}).  Hence $u'$ and $  (x^{2})^2$ form a basis of $\ann(\cE)$. The graph in the natural basis $\{x,y,x^{2},(x^{2})^2, u'\}$
    is 
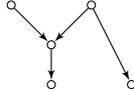
\[
 \begin{tikzpicture}[x=1.5pt, y=1.5pt]
 \tikzset{vertex/.style = {shape=circle,draw, inner sep=0pt, minimum size=3pt}}
 \tikzset{edge/.style = {->,> = latex'}}
 \node[vertex] (a) at  (-10,10) {};
 \node[vertex] (b) at  (10,10) {};
 \node[vertex] (c) at  (0,0) {};
\node[vertex] (d) at  (0,-10) {};
\node[vertex] (e) at  (20,-10) {};
 \draw[edge] (a) to (c);
 \draw[edge] (b) to (c);
  \draw[edge] (c) to (d);
   \draw[edge] (b) to (e);

 \end{tikzpicture}
 \]
Note that $\ann(\cE)$ is contained in $\cE^2$ and this is not valid for the only decomposable
nilpotent evolution algebra of type $[2,1,2]$.
        
\medskip
\item Finally,  if the type is $[2,1,1,1]$, and $\{x,y,z,u,v\}$ is  a natural basis with 
$\ann(\cE) = \espan{u,v}$, $\ann^2(\cE) = \espan{z,u,v}$, and $\ann^3(\cE) = \espan{y,z,u,v}$,
then 
\[
x^{[2]} \bydef x^2 = \alpha y +\beta z + p,
\] 
with $\alpha,\beta \in \FF$, $\alpha \neq 0$, 
  $p \in \ann(\cE)$,
\[
x^{[3]} \bydef (x^2)^2 = \alpha^2 y^2 + q = \mu z + q',
\] 
with  $ 0 \neq \mu \in \FF$ 
  and $q,q' \in \ann(\cE) $, and
\[
0 \neq x^{[4]} \bydef ( x^{[3]})^2 \in \ann(\cE).
\]
Thus $x x^{[r]} = 0$ for $ r=2,3,4$, $x^{[2]} x^{[3]} \in \FF z^2 = \FF  x^{[4]}$, and 
  hence $\cS = \espan{x,x^{[2]},x^{[3]}, x^{[4]}}$ is a subalgebra of $\cE$ with 
  $\cE = \cS + \ann(\cE)$, a contradiction   with Lemma \ref{lem:subalgebraann}.
Therefore, there are no indecomposable algebras with this type.
     \end{itemize}
   \end{proof}
   
\begin{remark}
Only the algebra of type $[2,1,2]$ appears in the classification in
  \cite{HA2}. It is denoted there as $\cE_{4,29}$. Also, the argument used in the proof above and Corollary \ref{cor:algebracuadann} show that a nilpotent evolution algebra of type $[n,1,m]$ is indecomposable if and only if $\ann(\cE)\subseteq \cE^2$.
\end{remark}

\begin{remark}
The argument in the proof above, to show that any nilpotent evolution algebra of type $[2,1,1,1]$ is decomposable, can be adjusted to prove that any nilpotent evolution algebra of dimension $n$ and type $[n_1,\ldots,n_r]$, with $r\geq 3$ and $2n_1>n-r+2$, is decomposable. Indeed, any $x\in\cE\setminus\ann^{r-1}(\cE)$ satisfies that the elements $x^{[2]},\ldots,x^{[r-1]}$ belong to $\cE^2$ and are linearly independent modulo $\ann(\cE)$. Hence, if $\cE$ were indecomposable, then $\ann(\cE)$ would be contained in $\cE^2$ (Corollary \ref{cor:algebracuadann}) and we would get:
\begin{multline*}
n_1+r-2=\dim_\FF\left(\ann(\cE)+\espan{x^{[2]},\ldots,x^{[r-1]}}\right)\leq \dim_\FF(\cE^2)\\
\leq \dim_\FF(\cE)-\dim_\FF\bigl(\ann(\cE)\bigr)=n-n_1,
\end{multline*}
so $2n_1\leq n-r+2$, a contradiction.
\end{remark}
 
 \medskip
 We are left with the classification of the five-dimensional, nilpotent, evolution algebras $\cE$, with $\dim(\ann(\cE)) = 1$. This condition already implies 
 $\cE$ to be indecomposable. The possible types of $\cE$, ordered lexicographically, are $[1,4]$, $[1,3,1]$, $[1,2,2]$, $[1,2,1,1]$, $[1,1,3]$, $[1,1,2,1]$, $[1,1,1,2]$, and $[1,1,1,1,1]$.
   
\bigskip
   We will classify first those algebras isomorphic to the algebras in  
   Theorems \ref{teo:threetypes} and \ref{teo:onetype}.

    \medskip
   In what follows, a dashed edge with weight $\alpha \in \FF$
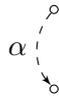
\[
\begin{tikzpicture}[x=1.5pt, y=1.5pt]

\tikzset{vertex/.style = {shape=circle,draw, inner sep=0pt, minimum size=3pt}}
\tikzset{edge/.style = {->,> = latex'}}
\node[vertex] (a) at  (0,30) {};
\node[vertex] (c) at  (0,10) {};

\draw[edge,dashed] (a) to[bend right=40] node [left]{$\alpha$}(c);
\end{tikzpicture}
\]
will indicate that there is no such edge if $\alpha = 0$, and that the edge is an 
usual edge with weight $\alpha$, if  $\alpha \neq 0$. As before, if no weight is attached to an edge, it means that the weight is $1$.

\begin{theorem}
   Let $\cE$ be an indecomposable, five-dimensional,  nilpotent, evolution algebra,  over an algebraically closed field of characteristic not two.
  \begin{romanenumerate}
    \item If the type is $[1,4]$, then $\cE$  is isomorphic to the algebra with graph
    
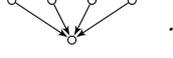
{\[
 \begin{tikzpicture}[x=1.5pt, y=1.5pt]

 \tikzset{vertex/.style = {shape=circle,draw, inner sep=0pt, minimum size=3pt}}
 \tikzset{edge/.style = {->,> = latex'}}
 \node[vertex] (a) at  (-15,10) {};
 \node[vertex] (b) at  (-5,10) {};
 \node[vertex] (c) at  (5,10) {};
  \node[vertex] (d) at  (15,10) {};
 \node[vertex] (e) at  (0,0) {};

 \draw[edge] (a) to (e);
 \draw[edge] (b) to (e);
 \draw[edge] (c) to (e);
  \draw[edge] (d) to (e);

 \end{tikzpicture}
\quad\raisebox{5pt}{\text{.}}
 \]}
    \item If the type is $[1,3,1]$, then $\cE$  is isomorphic to one and only one of the algebras with 
    the following  graphs
    {\[
 \vcenter
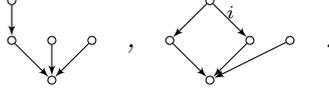
{\hbox{%
 \begin{tikzpicture}[x=1.5pt, y=1.5pt]

 \tikzset{vertex/.style = {shape=circle,draw, inner sep=0pt, minimum size=3pt}}
 \tikzset{edge/.style = {->,> = latex'}}
  \node[vertex] (a) at  (-10,20) {};
 \node[vertex] (b) at  (-10,10) {};
 \node[vertex] (c) at  (0,10) {};
 \node[vertex] (d) at  (10,10) {};
 \node[vertex] (e) at  (0,0) {};

  \draw[edge] (a) to (b);
 \draw[edge] (b) to (e);
 \draw[edge] (c) to (e);
 \draw[edge] (d) to (e);
 \end{tikzpicture}}}
  \quad\text{,}\quad
 \vcenter{\hbox{%
  \begin{tikzpicture}[x=1.5pt, y=1.5pt]
\tikzset{vertex/.style = {shape=circle,draw, inner sep=0pt, minimum size=3pt}}
\tikzset{edge/.style = {->,> = latex'}}
\node[vertex] (a) at  (0,20) {};
\node[vertex] (b) at  (-10,10) {};
\node[vertex] (c) at  (10,10) {};
\node[vertex] (d) at  (0,0) {};
\node[vertex] (e) at  (20,10) {};

\draw[edge] (a) to (b);
\draw[edge] (a) to node[near start,right =-2pt]{\tiny $i$} (c);
\draw[edge] (b) to (d);
\draw[edge] (c) to (d);
\draw[edge] (e) to (d);
\end{tikzpicture}}}
\quad\text{.}
 \]}
\item If the type is $[1,1,3]$, then $\cE$  is isomorphic to one of the algebras with
the following  graphs
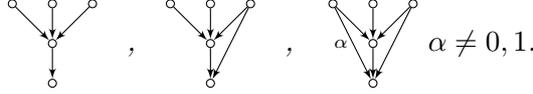
\[
 \vcenter{\hbox{%
 \begin{tikzpicture}[x=1.5pt, y=1.5pt]

 \tikzset{vertex/.style = {shape=circle,draw, inner sep=0pt, minimum size=3pt}}
 \tikzset{edge/.style = {->,> = latex'}}
 \node[vertex] (a) at  (-10,10) {};
 \node[vertex] (b) at  (0,10) {};
 \node[vertex] (c) at  (10,10) {};
 \node[vertex] (d) at  (0,0) {};
 \node[vertex] (e) at  (0,-10) {};
 \draw[edge] (a) to (d);
 \draw[edge] (b) to (d);
 \draw[edge] (c) to (d);
 \draw[edge] (d) to (e);
 \end{tikzpicture}}}
  \quad\text{,}\quad
 \vcenter{\hbox{%
 \begin{tikzpicture}[x=1.5pt, y=1.5pt]

 \tikzset{vertex/.style = {shape=circle,draw, inner sep=0pt, minimum size=3pt}}
 \tikzset{edge/.style = {->,> = latex'}}
 \node[vertex] (a) at  (-10,10) {};
 \node[vertex] (b) at  (0,10) {};
 \node[vertex] (c) at  (10,10) {};
 \node[vertex] (d) at  (0,0) {};
 \node[vertex] (e) at  (0,-10) {};
 \draw[edge] (a) to (d);
 \draw[edge] (b) to (d);
 \draw[edge] (c) to (d);
 \draw[edge] (d) to (e);
  \draw[edge] (c) to (e);
 \end{tikzpicture}}}
  \quad\text{,}\quad
 \vcenter{\hbox{%
  \begin{tikzpicture}[x=1.5pt, y=1.5pt]

 \tikzset{vertex/.style = {shape=circle,draw, inner sep=0pt, minimum size=3pt}}
 \tikzset{edge/.style = {->,> = latex'}}
 \node[vertex] (a) at  (-10,10) {};
 \node[vertex] (b) at  (0,10) {};
 \node[vertex] (c) at  (10,10) {};
 \node[vertex] (d) at  (0,0) {};
 \node[vertex] (e) at  (0,-10) {};
 \draw[edge] (a) to (d);
 \draw[edge] (b) to (d);
 \draw[edge] (c) to (d);
 \draw[edge] (d) to (e);
  \draw[edge] (c) to (e);
   \draw[edge] (a) to node [left=-2pt]{\tiny $\alpha$}(e);
 \end{tikzpicture}}}
 \ \text{$\alpha\ne 0,1$.}
 \]
Algebras with different graphs are not isomorphic, and two algebras with the third graph and parameters $\alpha, \alpha'$ are isomorphic if and only if
 $\alpha' \in \{\alpha, \alpha^{-1}, 1- \alpha, 1- \alpha^{-1}, (1- \alpha)^{-1},  (1- \alpha^{-1})^{-1}\}$.

    \item If the type is $[1,1,1,2]$, then $\cE$  is isomorphic to an algebra with one of
    the following  graphs
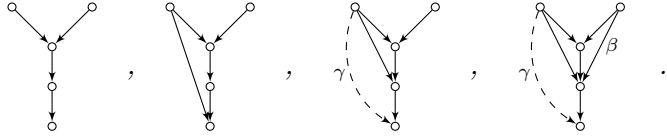
\[
 \vcenter{\hbox{%
 \begin{tikzpicture}[x=1.5pt, y=1.5pt]
 \tikzset{vertex/.style = {shape=circle,draw, inner sep=0pt, minimum size=3pt}}
 \tikzset{edge/.style = {->,> = latex'}}
 \node[vertex] (a) at  (-10,10) {};
 \node[vertex] (b) at  (10,10) {};
 \node[vertex] (c) at  (0,0) {};
 \node[vertex] (d) at  (0,-10) {};
  \node[vertex] (e) at  (0,-20) {};
 \draw[edge] (a) to (c);
 \draw[edge] (b) to (c);
 \draw[edge] (c) to (d);
  \draw[edge] (d) to (e);
 \end{tikzpicture}}}
  \quad\text{,}\quad
 \vcenter{\hbox{%
  \begin{tikzpicture}[x=1.5pt, y=1.5pt]
 \tikzset{vertex/.style = {shape=circle,draw, inner sep=0pt, minimum size=3pt}}
 \tikzset{edge/.style = {->,> = latex'}}
 \node[vertex] (a) at  (-10,10) {};
 \node[vertex] (b) at  (10,10) {};
 \node[vertex] (c) at  (0,0) {};
 \node[vertex] (d) at  (0,-10) {};
  \node[vertex] (e) at  (0,-20) {};
 \draw[edge] (a) to (c);
 \draw[edge] (b) to (c);
 \draw[edge] (c) to (d);
  \draw[edge] (d) to (e);
      \draw[edge] (a) to (e);
 \end{tikzpicture}}}
  \quad\text{,}\quad
 \vcenter{\hbox{%
  \begin{tikzpicture}[x=1.5pt, y=1.5pt]
 \tikzset{vertex/.style = {shape=circle,draw, inner sep=0pt, minimum size=3pt}}
 \tikzset{edge/.style = {->,> = latex'}}
 \node[vertex] (a) at  (-10,10) {};
 \node[vertex] (b) at  (10,10) {};
 \node[vertex] (c) at  (0,0) {};
 \node[vertex] (d) at  (0,-10) {};
  \node[vertex] (e) at  (0,-20) {};
 \draw[edge] (a) to (c);
 \draw[edge] (b) to (c);
 \draw[edge] (c) to (d);
  \draw[edge] (d) to (e);
   \draw[edge] (a) to (d);
    \draw[edge,dashed] (a) to[bend right=40] node [left=-2pt]{\tiny $\gamma$}(e);
 \end{tikzpicture}}}
  \quad\text{,}\quad
 \vcenter{\hbox{%
  \begin{tikzpicture}[x=1.5pt, y=1.5pt]
 \tikzset{vertex/.style = {shape=circle,draw, inner sep=0pt, minimum size=3pt}}
 \tikzset{edge/.style = {->,> = latex'}}
 \node[vertex] (a) at  (-10,10) {};
 \node[vertex] (b) at  (10,10) {};
 \node[vertex] (c) at  (0,0) {};
 \node[vertex] (d) at  (0,-10) {};
  \node[vertex] (e) at  (0,-20) {};
 \draw[edge] (a) to (c);
 \draw[edge] (b) to (c);
 \draw[edge] (c) to (d);
  \draw[edge] (d) to (e);
   \draw[edge] (a) to (d);
    \draw[edge,dashed] (a) to[bend right=40] node [left =-2pt]{\tiny $\gamma$}(e);
     \draw[edge] (b) to node [right =-2pt]{\tiny $\beta$}(d);
 \end{tikzpicture}}}
 \quad\text{.}
  \]
Algebras with different graphs are not isomorphic. 
   Algebras with the third graph and different values of the parameter  are not isomorphic.
   Finally, algebras with the fourth graph and parameters $(\beta,\gamma)$ and $(\beta',\gamma')$ are isomorphic if 
   and only if $(\beta',\gamma')=(\beta,\gamma)$ or $(\beta',\gamma')=(\beta^{-1},-\beta^{-3}\gamma)$.
\end{romanenumerate}
\end{theorem}

 \begin{proof}
 Type $[1,4] $ is covered by Theorem \ref{teo:threetypes} (i) and type $[1,3,1]$ by Corollary \ref{cor:type[1,n1]}.
 
 Any algebra of type $[1,1,3]$ is isomorphic to an algebra $\cE(\cU, b,g)$, with $\dim(\cU) = 3$, as in Theorem \ref{teo:threetypes}(ii), and we may change the symmetric endomorphism  $g$ by $\mu g + \nu \id$ for $\mu, \nu \in \FF, \mu \neq 0$.

There are then three possibilities which give nonisomorphic algebras:

 \begin{itemize}

  \item  $g$ has a unique eigenvalue $\nu$ with multiplicity $3$. In this case, replacing $g$ by $g-\nu\id$, we may assume $g=0$ and get the first graph in item (iii).

  \item $g$ has two eigenvalues, one   with multiplicity $2$. Then we may assume that the eingenvalues are $0$, with multiplicity $2$, and $1$, and we get the second graph in item (iii).

  \item $g$ has three different eigenvalues, which we may assume to be $0,1,\alpha$, 
  $\alpha \in \FF \setminus\{0,1\}$.  The algebras with parameters $\alpha,\alpha' \in \FF \setminus\{0,1\}$ 
  are isomorphic if and only if there are scalars $\mu, \nu \in \FF, \mu \neq 0$, such that 
  $\{\nu, \mu + \nu, \mu \alpha  + \nu  \} = \{0,1,\alpha' \}$
  and the result in item (iii) follows. 
\end{itemize}

 Finally, any algebra of type $[1,1,1,2]$ is  isomorphic to an algebra $\cE(\cU, b,f,g)$  as in Theorem \ref{teo:threetypes}(iii). We may change $(f,g)$ by $(\mu f, \mu^3 g
 + \nu \id)$, with $\mu, \nu \in \FF, \mu \neq 0$.
 
 If $f = 0$ we get the first two possibilities, according to $g$ having only one eigenvalue (with multiplicity $2$), or two different eigenvalues.
 
 If $f \neq 0$ but one of its eigenvalues is $0$, then we get the third possibility, as we can assume that 
 the nonzero eigenvalue is $1$. Finally, if $0$ is not an 
eigenvalue of $f$, we may assume that the eigenvalues are $1$, $\beta $, with $0 \neq \beta \in \FF$. Then we may 
take $g$ with corresponding eigenvalues $\gamma$ and $0$. 
Two such algebras with parameters $(\beta, \gamma) $ and $(\beta', \gamma') $ are isomorphic if and only if there 
are scalars $\mu, \nu \in \FF, \mu \neq 0$, such that 
either $(1,\beta') = \mu (1,\beta) $ and $(\gamma',0) =\mu^3(\gamma,0) + \nu (1,1)$ or $(1,\beta') = \mu (\beta,1) $ and $(\gamma',0) =\mu^3(0,\gamma) + \nu (1,1)$, 
whence the result.
 \end{proof}
 
 \medskip
 
 The remaining types: $[1,2,2], [1,2,1,1], [1,1,2,1]$ and $[1,1,1,1,1]$, will be treated separately.
 
 Without further mention, the following fact will be used. Given a nilpotent evolution algebra of type 
 $[n_1, \ldots, n_r]$ and a natural basis $B =\{x_1, \ldots, x_n\}$ of $\cE$ with 
 $x_1, \ldots, x_{n_r} \in \ann^r(\cE) \setminus \ann^{r-1}(\cE)$, and 
 $x_{n_{r} +1}, \ldots, x_{n} \in \ann^{r-1}(\cE)$, then if we pick any other 
 natural basis of $\ann^{r-1}(\cE)$: $\{ y_{n_{r} +1}, \ldots, y_{n} \}$, the new basis 
 $\{ x_1, \ldots, x_{n_r}, y_{n_{r} +1}, \ldots, y_{n} \}$ is again a natural basis of 
 $\cE$.
 
 \begin{theorem}
   Let $\cE$ be a  nilpotent evolution algebra of type $[1,2,2]$, over an algebraically closed field of 
   characteristic not two. Then $\cE$ is isomorphic to an algebra with one of
   the following graphs:
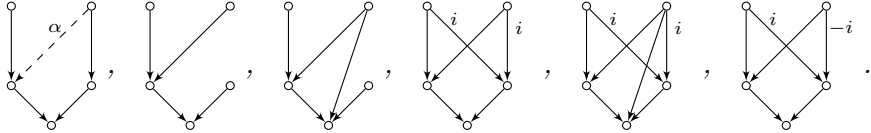
\[
 \vcenter{\hbox{%
  \begin{tikzpicture}[x=1.5pt, y=1.5pt]
 \tikzset{vertex/.style = {shape=circle,draw, inner sep=0pt, minimum size=3pt}}
 \tikzset{edge/.style = {->,> = latex'}}
 \node[vertex] (a) at  (-10,30) {};
 \node[vertex] (b) at  (10,30) {};
  \node[vertex] (c) at  (-10,10) {};
 \node[vertex] (d) at  (10,10) {};
 \node[vertex] (e) at  (0,0) {};
 
 \draw[edge] (a) to (c);
 \draw[edge] (b) to (d);
 \draw[edge,dashed] (b) to node [near start,left=-1pt]{\tiny $\alpha$} (c);
 \draw[edge] (d) to (e);
 \draw[edge] (c) to (e);
 
 \end{tikzpicture}}}
  \ \text{,}\quad
 \vcenter{\hbox{%
  \begin{tikzpicture}[x=1.5pt, y=1.5pt]
 \tikzset{vertex/.style = {shape=circle,draw, inner sep=0pt, minimum size=3pt}}
 \tikzset{edge/.style = {->,> = latex'}}
 \node[vertex] (a) at  (-10,30) {};
 \node[vertex] (b) at  (10,30) {};
  \node[vertex] (c) at  (-10,10) {};
 \node[vertex] (d) at  (10,10) {};
 \node[vertex] (e) at  (0,0) {};
 
 \draw[edge] (a) to (c);
 \draw[edge] (b) to (c);
 \draw[edge] (d) to (e);
 \draw[edge] (c) to (e);
 \end{tikzpicture}}}
  \ \text{,}\quad
 \vcenter{\hbox{%
  \begin{tikzpicture}[x=1.5pt, y=1.5pt]
 \tikzset{vertex/.style = {shape=circle,draw, inner sep=0pt, minimum size=3pt}}
 \tikzset{edge/.style = {->,> = latex'}}
 \node[vertex] (a) at  (-10,30) {};
 \node[vertex] (b) at  (10,30) {};
  \node[vertex] (c) at  (-10,10) {};
 \node[vertex] (d) at  (10,10) {};
 \node[vertex] (e) at  (0,0) {};
 
 \draw[edge] (a) to (c);
 \draw[edge] (b) to (c);
  \draw[edge] (b) to (e);
 \draw[edge] (d) to (e);
 \draw[edge] (c) to (e);

 \end{tikzpicture}}}
  \ \text{,}\quad
 \vcenter{\hbox{%
  \begin{tikzpicture}[x=1.5pt, y=1.5pt]
 \tikzset{vertex/.style = {shape=circle,draw, inner sep=0pt, minimum size=3pt}}
 \tikzset{edge/.style = {->,> = latex'}}
 \node[vertex] (a) at  (-10,30) {};
 \node[vertex] (b) at  (10,30) {};
  \node[vertex] (c) at  (-10,10) {};
 \node[vertex] (d) at  (10,10) {};
 \node[vertex] (e) at  (0,0) {};
 
 \draw[edge] (a) to (c);
 \draw[edge] (a) to node[very near start,right]{\tiny $i$} (d);
 \draw[edge] (b) to (c);
 \draw[edge] (b) to node[near start,right=-1pt]{\tiny $i$} (d);
 \draw[edge] (d) to (e);
 \draw[edge] (c) to (e);

 \end{tikzpicture}}}
  \ \text{,}\quad
 \vcenter{\hbox{%
  \begin{tikzpicture}[x=1.5pt, y=1.5pt]
 \tikzset{vertex/.style = {shape=circle,draw, inner sep=0pt, minimum size=3pt}}
 \tikzset{edge/.style = {->,> = latex'}}
 \node[vertex] (a) at  (-10,30) {};
 \node[vertex] (b) at  (10,30) {};
  \node[vertex] (c) at  (-10,10) {};
 \node[vertex] (d) at  (10,10) {};
 \node[vertex] (e) at  (0,0) {};
 
 \draw[edge] (a) to (c);
  \draw[edge] (a) to node[very near start,right]{\tiny $i$} (d);
 \draw[edge] (b) to (c);
  \draw[edge] (b) to node[near start,right=-1]{\tiny $i$} (d);
  \draw[edge] (b) to (e);
 \draw[edge] (d) to (e);
 \draw[edge] (c) to (e);

 \end{tikzpicture}}}
  \ \text{,}\quad
 \vcenter{\hbox{%
  \begin{tikzpicture}[x=1.5pt, y=1.5pt]
 \tikzset{vertex/.style = {shape=circle,draw, inner sep=0pt, minimum size=3pt}}
 \tikzset{edge/.style = {->,> = latex'}}
 \node[vertex] (a) at  (-10,30) {};
 \node[vertex] (b) at  (10,30) {};
  \node[vertex] (c) at  (-10,10) {};
 \node[vertex] (d) at  (10,10) {};
 \node[vertex] (e) at  (0,0) {};
 
  \draw[edge] (a) to (c);
 \draw[edge] (a) to node[very near start,right]{\tiny $i$} (d);
 \draw[edge] (b) to (c);
 \draw[edge] (b) to node[near start,right=-3pt]{\tiny $-i$} (d);
 \draw[edge] (d) to (e);
 \draw[edge] (c) to (e);

 \end{tikzpicture}}} 
 \text{.}
 \]
Algebras with different graphs are not isomorphic, and two algebras with the first graph and parameters 
$\alpha, \alpha'$ are isomorphic if and only if 
$\alpha' = \alpha $ or $\alpha' = -\alpha$.

 \end{theorem}

 \begin{proof}
 Let $\{x,y,u,v,s\}$ be  a natural basis with $\ann(\cE) = \FF s$ and $\ann^2(\cE)= \espan{u,v,s}$. 
 Since $\ann^2(\cE) $ is of type $[1,2]$ we may assume that $ u^2 = v^2 = s$. Let us consider, as in 
 Section \ref{se: upperseries}, the  subspaces $\cU_1 = \FF s$, $\cU_2 = \FF u \oplus  \FF v$, and
 $\cU_3 = \FF x \oplus  \FF y$. Now,
\begin{equation}\label{eq:x2y2}
x^2 = a +\mu s, \; \; y^2 = b +\nu  s,
\end{equation}
with $0\neq a, b \in \cU_2$ and $\mu, \nu \in \FF$. Then, either $ a^2 = b^2 = 0$ ( that is, $(x^2)^2 = (y^2)^2 = 0$), or we may assume that $a^2 \neq 0$.
 
 In the latter case: $a^2 \neq 0$, we may change $u $ by $a + \mu s $ and $v$ by an element in $\cU_2$ orthogonal to $a$.  Thus, we may assume that $x^2 = u$ and $y^2 =
 \alpha u + \beta v +\gamma s$ with $(\alpha,\beta) \neq (0,0)$.
 
 If $\beta \neq 0$, the dimension of $\cE^2$ is $3$ and changing $v$ by $v + \beta^{-1} \gamma  s$, and scaling $y$, we may assume $y^2 = \alpha u + v$, thus getting the graph 
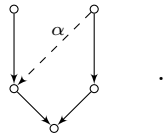
\[ 
   \begin{tikzpicture}[x=1.5pt, y=1.5pt]
 \tikzset{vertex/.style = {shape=circle,draw, inner sep=0pt, minimum size=3pt}}
 \tikzset{edge/.style = {->,> = latex'}}
 \node[vertex] (a) at  (-10,30) {};
 \node[vertex] (b) at  (10,30) {};
  \node[vertex] (c) at  (-10,10) {};
 \node[vertex] (d) at  (10,10) {};
 \node[vertex] (e) at  (0,0) {};
 
 \draw[edge] (a) to (c);
 \draw[edge] (b) to (d);
 \draw[edge,dashed] (b) to node [near start,left=-1pt]{\tiny $\alpha$} (c);
 \draw[edge] (d) to (e);
 \draw[edge] (c) to (e);
 
 \end{tikzpicture}\qquad\raisebox{20pt}{\text{.}}
  \]
 
 If this algebra were isomorphic to an algebra with the same graph but with parameter $\alpha'$, there woud be another natural basis 
 $\{x',y',u',v',s'\}$ with (see Corollary \ref{cor:basechange}): 
  \[
 x' =\mu_{11} x + \mu_{12} y + \nu_1 s,   \   y' =\mu_{21} x + \mu_{22} y + \nu_2 s 
  \]
    \[
  u' = (x')^2, \ v' = (y')^2 -\alpha' u',  \  s' = (u')^2 = (v')^2
 \]
 From $x'y' = 0$ we get $\mu_{11} \mu_{21} + \mu_{12} \mu_{22} = 0 = \mu_{12} \mu_{21}, $ so either $\mu_{12} = 0$ or $\mu_{21} = 0$.
 
 If $\mu_{12} = 0$ we get $\mu_{11}\neq 0 \neq \mu_{22} $ and $\mu_{21} = 0$.  From $u'v' = 0$ we quickly chek that $(\alpha')^2 = \alpha^2$,   and the same happens if $\mu_{22} = 0$.
 
 If $\beta = 0$ then $\dim(\cE^2) = 2$ and, after scaling $y$, we get the graph
 
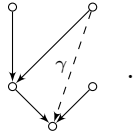
\[
    \begin{tikzpicture}[x=1.5pt, y=1.5pt]
 \tikzset{vertex/.style = {shape=circle,draw, inner sep=0pt, minimum size=3pt}}
 \tikzset{edge/.style = {->,> = latex'}}
 \node[vertex] (a) at  (-10,30) {};
 \node[vertex] (b) at  (10,30) {};
  \node[vertex] (c) at  (-10,10) {};
 \node[vertex] (d) at  (10,10) {};
 \node[vertex] (e) at  (0,0) {};
 
 \draw[edge] (a) to (c);
  \draw[edge] (b) to (c);
 \draw[edge,dashed] (b) to node [left=-2pt]{\tiny $\gamma$} (e);
 \draw[edge] (d) to (e);
 \draw[edge] (c) to (e);
 
 \end{tikzpicture}\quad\raisebox{20pt}{\text{.}}
  \]
  If $\gamma \neq 0$, in the new natural basis $\{\sqrt{\gamma^{-1}} x, \sqrt{\gamma^{-1}}y, \gamma^{-1} u, \gamma^{-1} v, \gamma^{-2} s\}$ we get the graph above with $\gamma = 1$. 
  Note that for $\gamma = 0$, $\dim(\cU_3 \oplus \cU_1)^2 = 1$, and for $\gamma = 1$, 
  $\dim(\cU_3 \oplus \cU_1)^2 = 2$, so we obtain nonisomorphic algebras.
  
  Finally, if $ a^2 = b^2 = 0$ in \eqref{eq:x2y2}, then we may assume $x^2 = u +iv$, with $i^2 = -1$, and $y^2 = \epsilon u +\delta v + \gamma s$, with $\epsilon, \delta \neq 0$, $\epsilon^2 +\delta^2 = 0$. Scaling $ y$ we may assume $\epsilon = 1$ and $\delta = \pm i$.
  
  If $\delta = i$, $\dim(\cU_3 \oplus \cU_1)^2 = 1$ for $\gamma = 0$ and we obtain the fourth graph, and $\dim(\cU_3 \oplus \cU_1)^2 = 2$ if $\gamma \neq 0$. In the latest case it is easy to get a new natural basis with $\gamma = 1$, and we obtain the fifth graph. Here $\dim\cE^2=2$ holds.
  
    If $\delta = -i$, changing $ u$ by $u+ \frac{1}{2} \gamma s$ and $ v$ by $v + \frac{i}{2} \gamma s$ we may assume $\gamma = 0$, thus obtaining the last graph. In this case $\dim\cE^2=3$.
  \end{proof}
  
   \begin{theorem}
   Let $\cE$ be a  nilpotent evolution algebra of type $[1,2,1,1]$, over an algebraically closed field of characteristic not two. Then $\cE$ is isomorphic to an algebra with one of
   the following graphs:
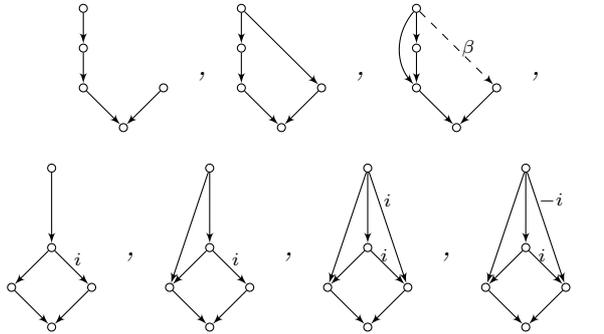
\[
 \vcenter{\hbox{%
  \begin{tikzpicture}[x=1.5pt, y=1.5pt]
 \tikzset{vertex/.style = {shape=circle,draw, inner sep=0pt, minimum size=3pt}}
 \tikzset{edge/.style = {->,> = latex'}}
  \node[vertex] (a) at  (-10,30) {};
 \node[vertex] (b) at  (-10,20) {};
  \node[vertex] (c) at  (-10,10) {};
 \node[vertex] (d) at  (10,10) {};
 \node[vertex] (e) at  (0,0) {};

 \draw[edge] (a) to (b);
 \draw[edge] (b) to (c);
 \draw[edge] (c) to (e);
 \draw[edge] (d) to (e);
 \end{tikzpicture}}}
  \quad\text{,}\quad
%
%
%
 %
 \vcenter{\hbox{%
 \begin{tikzpicture}[x=1.5pt, y=1.5pt]
 \tikzset{vertex/.style = {shape=circle,draw, inner sep=0pt, minimum size=3pt}}
 \tikzset{edge/.style = {->,> = latex'}}
  \node[vertex] (a) at  (-10,30) {};
 \node[vertex] (b) at  (-10,20) {};
  \node[vertex] (c) at  (-10,10) {};
 \node[vertex] (d) at  (10,10) {};
 \node[vertex] (e) at  (0,0) {};

 \draw[edge] (a) to (b);
  \draw[edge] (a) to (d);
 \draw[edge] (b) to (c);
 \draw[edge] (c) to (e);
 \draw[edge] (d) to (e);
 \end{tikzpicture}}}
  \quad\text{,}\quad
 \vcenter{\hbox{%
 \begin{tikzpicture}[x=1.5pt, y=1.5pt]
 \tikzset{vertex/.style = {shape=circle,draw, inner sep=0pt, minimum size=3pt}}
 \tikzset{edge/.style = {->,> = latex'}}
  \node[vertex] (a) at  (-10,30) {};
 \node[vertex] (b) at  (-10,20) {};
  \node[vertex] (c) at  (-10,10) {};
 \node[vertex] (d) at  (10,10) {};
 \node[vertex] (e) at  (0,0) {};

 \draw[edge] (a) to (b);
 \draw[edge,dashed] (a) to node[right=-2pt]{\tiny $\beta$} (d);
 \draw[edge] (a) to[bend right=40] (c);
 \draw[edge] (b) to (c);
 \draw[edge] (c) to (e);
 \draw[edge] (d) to (e);
 \end{tikzpicture}}}
  \quad\text{,}\quad
  \]
 \[
 \vcenter{\hbox{%
   \begin{tikzpicture}[x=1.5pt, y=1.5pt]
\tikzset{vertex/.style = {shape=circle,draw, inner sep=0pt, minimum size=3pt}}
\tikzset{edge/.style = {->,> = latex'}}
\node[vertex] (a) at  (0,40) {};
\node[vertex] (b) at  (0,20) {};
\node[vertex] (c) at  (-10,10) {};
\node[vertex] (d) at  (10,10) {};
\node[vertex] (e) at  (0,0) {};

\draw[edge] (a) to (b);
\draw[edge] (b) to node[near start,right]{\tiny $i$} (d);
\draw[edge] (b) to (c);
\draw[edge] (c) to (e);
\draw[edge] (d) to (e);
\end{tikzpicture}}}
  \quad\text{,}\quad
 \vcenter{\hbox{%
  \begin{tikzpicture}[x=1.5pt, y=1.5pt]
\tikzset{vertex/.style = {shape=circle,draw, inner sep=0pt, minimum size=3pt}}
\tikzset{edge/.style = {->,> = latex'}}
\node[vertex] (a) at  (0,40) {};
\node[vertex] (b) at  (0,20) {};
\node[vertex] (c) at  (-10,10) {};
\node[vertex] (d) at  (10,10) {};
\node[vertex] (e) at  (0,0) {};

\draw[edge] (a) to (b);
\draw[edge] (a) to (c);
\draw[edge] (b) to node[near start,right]{\tiny $i$} (d);
\draw[edge] (b) to (c);
\draw[edge] (c) to (e);
\draw[edge] (d) to (e);
\end{tikzpicture}}}
  \quad\text{,}\quad
 \vcenter{\hbox{%
  \begin{tikzpicture}[x=1.5pt, y=1.5pt]
\tikzset{vertex/.style = {shape=circle,draw, inner sep=0pt, minimum size=3pt}}
\tikzset{edge/.style = {->,> = latex'}}
\node[vertex] (a) at  (0,40) {};
\node[vertex] (b) at  (0,20) {};
\node[vertex] (c) at  (-10,10) {};
\node[vertex] (d) at  (10,10) {};
\node[vertex] (e) at  (0,0) {};

\draw[edge] (a) to (b);
\draw[edge] (a) to (c);
\draw[edge] (a) to node[near start,right=-2pt]{\tiny $i$} (d);
\draw[edge] (b) to node[very near start,right=-2pt]{\tiny $i$} (d);
\draw[edge] (b) to (c);
\draw[edge] (c) to (e);
\draw[edge] (d) to (e);
\end{tikzpicture}}}
  \quad\text{,}\quad
 \vcenter{\hbox{%
  \begin{tikzpicture}[x=1.5pt, y=1.5pt]
\tikzset{vertex/.style = {shape=circle,draw, inner sep=0pt, minimum size=3pt}}
\tikzset{edge/.style = {->,> = latex'}}
\node[vertex] (a) at  (0,40) {};
\node[vertex] (b) at  (0,20) {};
\node[vertex] (c) at  (-10,10) {};
\node[vertex] (d) at  (10,10) {};
\node[vertex] (e) at  (0,0) {};

\draw[edge] (a) to (b);
\draw[edge] (a) to (c);
\draw[edge] (a) to node[near start,right=-3pt]{\tiny $-i$} (d);
\draw[edge] (b) to node[very near start,right=-2pt]{\tiny $i$} (d);
\draw[edge] (b) to (c);
\draw[edge] (c) to (e);
\draw[edge] (d) to (e);
\end{tikzpicture}}}
  \quad\text{.}\quad
 \]
Algebras with different graphs are not isomorphic, and two algebras with the third graph, and parameters 
$\beta, \beta' \in \FF$, are isomorphic if and only if either $\beta'=\beta$ or $\beta'=-\beta$.
 \end{theorem}
\begin{proof} 
 We have two possibilities, depending on $\ann^3(\cE)$, which is of type $[1,2,1]$ and hence isomorphic (Theorem \ref{teo:indecomposabledimfour}) to an algebra with one of these graphs
 \begin{equation}\label{eq:1}
  \vcenter{\hbox{%
 \begin{tikzpicture}[x=1.5pt, y=1.5pt]
 \tikzset{vertex/.style = {shape=circle,draw, inner sep=0pt, minimum size=3pt}}
 \tikzset{edge/.style = {->,> = latex'}}
 \node[vertex] (a) at  (-10,20) {};
 \node[vertex] (b) at  (-10,10) {};
 \node[vertex] (c) at  (10,10) {};
 \node[vertex] (d) at  (0,0) {};

 \draw[edge] (a) to (b);
 \draw[edge] (b) to (d);
 \draw[edge] (c) to (d);

 \end{tikzpicture}}}
  \quad\text{or}\quad
 \vcenter{\hbox{%
  \begin{tikzpicture}[x=1.5pt, y=1.5pt]

\tikzset{vertex/.style = {shape=circle,draw, inner sep=0pt, minimum size=3pt}}
\tikzset{edge/.style = {->,> = latex'}}
\node[vertex] (a) at  (0,20) {};
\node[vertex] (b) at  (-10,10) {};
\node[vertex] (c) at  (10,10) {};
\node[vertex] (d) at  (0,0) {};

\draw[edge] (a) to (b);
\draw[edge] (a) to node[near start,right]{$i$} (c);
\draw[edge] (b) to (d);
\draw[edge] (c) to (d);

\end{tikzpicture}}}
\quad\text{.}
 \end{equation}
 Assume  first that $\ann^3(\cE)$ is of the first class in (\ref{eq:1}). Then there is a natural basis  $\{x,y,u,v,s\}$ with graph
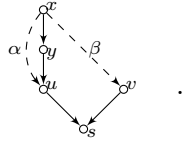
\[
\begin{tikzpicture}[x=1.5pt, y=1.5pt]
 \tikzset{vertex/.style = {shape=circle,draw, inner sep=0pt, minimum size=3pt}}
 \tikzset{edge/.style = {->,> = latex'}}
  \node[vertex] (a) at  (-10,30) {};
  \node at (-8,31){\tiny$x$};
 \node[vertex] (b) at  (-10,20) {};
   \node at (-8,19){\tiny$y$};
  \node[vertex] (c) at  (-10,10) {};
   \node at (-8,11){\tiny$u$};
 \node[vertex] (d) at  (10,10) {};
    \node at (12,11){\tiny$v$};
 \node[vertex] (e) at  (0,0) {};
    \node at (2,-1){\tiny$s$};

 \draw[edge] (a) to (b);
 \draw[edge,dashed] (a) to node[right=-2pt]{\tiny $\beta$} (d);
 \draw[edge,dashed] (a) to[bend right=40] node[left=-2pt]{\tiny$\alpha$} (c);
 \draw[edge] (b) to (c);
 \draw[edge] (c) to (e);
 \draw[edge] (d) to (e);
 \end{tikzpicture}\quad\raisebox{20pt}{\text{.}}
 \]
If  $\{x',y',u',v',s'\}$ is another natural basis with the same graph but with parameters 
 $\alpha', \beta'$, then (see Corollary  \ref{cor:basechange}) 
 there are scalars $\epsilon_1, \epsilon_2, \nu_1, \nu_2 \in \FF, \epsilon_1 \neq 0 \neq \epsilon_2$, 
 such that $x' = \epsilon_1 x + \nu_1 s$, $y' = \epsilon_2 y
 + \nu_2 s$, and then $u' = (y')^2 = \epsilon_2^2 u$. 
 Also $u'v' = 0$, so $uv' = 0$ and hence $v' = \epsilon_3 v + \nu_3 s$,
 with $\epsilon_3, \nu_3 \in \FF$, $\epsilon_3 \neq 0$. 
 
 Now $s'=(u')^2=(v')^2$, so $\epsilon_2^4=\epsilon_3^2$ and $\epsilon_3=\pm\epsilon_2^2$.
 
 From $(x')^2=y'+\alpha'u'+\beta'v'$ we get $\epsilon_2=\epsilon_1^2$, $\epsilon_1^2\alpha=\epsilon_2^2\alpha'$,
 $\epsilon_1^2\beta=\epsilon_3\beta'$, and $\nu_2+\beta'\nu_3=0$. We conclude that $\alpha=\epsilon_2\alpha'$ and 
 $\beta=\pm \epsilon_2\beta'$.
 Therefore, the parameters $(\alpha, \beta)$ can be taken to be $(0,0)$, $(0,1)$, or $(1,\beta)$ with $\beta \in\FF$, and the algebras with parameters
 $(1,\beta)$ and $(1,\beta')$ are  isomorphic if and only if 
$\beta' \in\{\beta,-\beta\}$. 
 
 Now assume that  $\ann^3(\cE)$ is of the second class in \eqref{eq:1}. Then there is a natural basis  $\{x,y,u,v,s\}$ of $\cE$ with graph
   
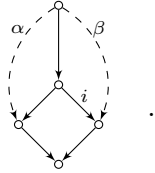
\[
    \begin{tikzpicture}[x=1.5pt, y=1.5pt]
\tikzset{vertex/.style = {shape=circle,draw, inner sep=0pt, minimum size=3pt}}
\tikzset{edge/.style = {->,> = latex'}}
\node[vertex] (a) at  (0,40) {};
\node[vertex] (b) at  (0,20) {};
\node[vertex] (c) at  (-10,10) {};
\node[vertex] (d) at  (10,10) {};
\node[vertex] (e) at  (0,0) {};

\draw[edge] (a) to (b);
 \draw[edge,dashed] (a) to [bend left =40]node[ near start, right=-2pt]{\tiny $\beta$} (d);
 \draw[edge,dashed] (a) to [bend right =40] node[ near start, left=-2pt]{\tiny $\alpha$} (c);
\draw[edge] (b) to node[near start, right]{\tiny $i$}(d);
\draw[edge] (b) to (c);
\draw[edge] (c) to (e);
\draw[edge] (d) to (e);
\end{tikzpicture}
\quad\raisebox{20pt}{\text{.}}
   \]
   That is, $x^2 = y +\alpha u+ \beta v$, $y^2  = u + iv$, $u^2 = v^2 = s$.
   
   The case $\alpha =\beta = 0$ happens if and only if  $(\cU_4\oplus\cU_1)^2 \subseteq\cU_3\oplus\cU_1. $
   
   If $\alpha \neq 0$ we may take the  new natural basis 
   $\{\sqrt{\alpha} x, \alpha y,  \alpha^2 u, \alpha^2 v, \alpha^4 s \}$ and hence assume that $\alpha = 1$. 
Similarly, if $\beta \neq 0$,
    we may take the  new natural basis $\{\sqrt{\beta} x, \beta y,  \beta^2 iv, - \beta^2 iu, -\beta^4 s \}$ and again assume that $\alpha = 1$.
    
    Therefore, either $\alpha =\beta = 0$ or there is a natural basis $\{x,y,u,v,s\}$
    with graph
    \[
      \begin{tikzpicture}[x=1.5pt, y=1.5pt]
\tikzset{vertex/.style = {shape=circle,draw, inner sep=0pt, minimum size=3pt}}
\tikzset{edge/.style = {->,> = latex'}}
\node[vertex] (a) at  (0,40) {};
\node[vertex] (b) at  (0,20) {};
\node[vertex] (c) at  (-10,10) {};
\node[vertex] (d) at  (10,10) {};
\node[vertex] (e) at  (0,0) {};

\draw[edge] (a) to (b);
\draw[edge] (a) to (c);
 \draw[edge,dashed] (a) to [bend left =40]node[ near start, right=-2pt]{\tiny $\beta$} (d);
\draw[edge] (b) to node[near start,right]{\tiny $i$} (d);
\draw[edge] (b) to (c);
\draw[edge] (c) to (e);
\draw[edge] (d) to (e);
\end{tikzpicture}
\quad\raisebox{20pt}{\text{.}}
\]
The natural linear  map 
$\ann^2(\cE) \hookrightarrow \ann^3(\cE) \longrightarrow \ann^3(\cE) / (\cU_3 \oplus \cU_1)$ induces
a linear isomorphism $\phi: \ann^2(\cE)  / \ann^1(\cE)  \longrightarrow \ann^3(\cE) / (\cU_3 \oplus \cU_1)$. 

Consider too the map $ \varphi: \cU_4 \oplus \cU_1  \longrightarrow \ann^3(\cE)/ (\cU_3 \oplus \cU_1)$,  
$z \mapsto z^2 +(\cU_3 \oplus \cU_1)$,
and the map $\psi: \ann^2(\cE)  / \ann^1(\cE)  \longrightarrow \ann^1(\cE)$, $z + \ann^1(\cE) \mapsto z^2$. 
The composition 
$\Phi = \psi \circ \phi^{-1} \circ \varphi $ is given by
$
\Phi: \cU_4 \oplus \cU_1 \longrightarrow \cU_1$, $\epsilon x + \delta s \mapsto
(\epsilon^2(u + \beta v))^2 = \epsilon^4(1 +\beta^2)s$.

Thus $\Phi  \neq 0$ if and only if $\beta = \pm i$. Also $\cE^2 \cap (\cU_3 \oplus \cU_1) $ has dimension $2$ if and only if $ \beta = i$. This 
shows that the algebras with parameter $\beta $ equal to $i$ and $-i$ are not isomorphic, and they are not isomorphic to any algebra with parameter
$\beta \neq \pm i$.

Finally, if $\beta \neq \pm i, $ take $ \mu = (1 +i\beta )(\sqrt{1 +\beta^2})^{-1}$, then the new natural  basis $\{\sqrt{\mu^{-1}} x, \mu^{-1} y, \mu^{-1}(u + \beta v),
\mu^{-1}(-\beta u + v), \mu^{-2}(1 +\beta^2) s\}$ has the above graph with $\beta = 0$.
\end{proof}
  
  \medskip
  
  The proof of the classification for types $[1,1,2,1]$ and $[1,1,1,1,1]$ follow similar arguments and will be omitted.
  
   \begin{theorem}
   Let $\cE$ be a  nilpotent evolution algebra of type $[1,1,2,1]$, over an algebraically closed field of 
   characteristic not two. Then $\cE$ is isomorphic to an algebra with one of
   the following graphs
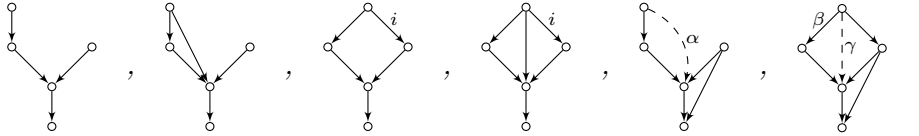
\[
 \vcenter{\hbox{%
  \begin{tikzpicture}[x=1.5pt, y=1.5pt]
 \tikzset{vertex/.style = {shape=circle,draw, inner sep=0pt, minimum size=3pt}}
 \tikzset{edge/.style = {->,> = latex'}}
 \node[vertex] (a) at  (-10,20) {};
  \node[vertex] (b) at  (-10,10) {};
 \node[vertex] (c) at  (10,10) {};
 \node[vertex] (d) at  (0,0) {};
 \node[vertex] (e) at  (0,-10) {};
 \draw[edge] (a) to (b);
 \draw[edge] (b) to (d);
 \draw[edge] (c) to (d);
 \draw[edge] (d) to (e);
 \end{tikzpicture}}}
  \quad\text{,}\quad
 \vcenter{\hbox{%
 \begin{tikzpicture}[x=1.5pt, y=1.5pt]
 \tikzset{vertex/.style = {shape=circle,draw, inner sep=0pt, minimum size=3pt}}
 \tikzset{edge/.style = {->,> = latex'}}
 \node[vertex] (a) at  (-10,20) {};
  \node[vertex] (b) at  (-10,10) {};
 \node[vertex] (c) at  (10,10) {};
 \node[vertex] (d) at  (0,0) {};
 \node[vertex] (e) at  (0,-10) {};
 \draw[edge] (a) to (b);
  \draw[edge] (a) to (d);
 \draw[edge] (b) to (d);
 \draw[edge] (c) to (d);
 \draw[edge] (d) to (e);
 \end{tikzpicture}}}
  \quad\text{,}\quad
 \vcenter{\hbox{%
 \begin{tikzpicture}[x=1.5pt, y=1.5pt]

\tikzset{vertex/.style = {shape=circle,draw, inner sep=0pt, minimum size=3pt}}
\tikzset{edge/.style = {->,> = latex'}}
\node[vertex] (a) at  (0,20) {};
\node[vertex] (b) at  (-10,10) {};
\node[vertex] (c) at  (10,10) {};
\node[vertex] (d) at  (0,0) {};
\node[vertex] (e) at  (0,-10) {};

\draw[edge] (a) to (b);
\draw[edge] (a) to node[near start,right]{\tiny $i$} (c);
\draw[edge] (b) to (d);
\draw[edge] (c) to (d);
\draw[edge] (d) to (e);

\end{tikzpicture}}}
  \quad\text{,}\quad
 \vcenter{\hbox{%
 \begin{tikzpicture}[x=1.5pt, y=1.5pt]

\tikzset{vertex/.style = {shape=circle,draw, inner sep=0pt, minimum size=3pt}}
\tikzset{edge/.style = {->,> = latex'}}
\node[vertex] (a) at  (0,20) {};
\node[vertex] (b) at  (-10,10) {};
\node[vertex] (c) at  (10,10) {};
\node[vertex] (d) at  (0,0) {};
\node[vertex] (e) at  (0,-10) {};
\draw[edge] (a) to (b);
\draw[edge] (a) to (d);
\draw[edge] (a) to node[near start,right]{\tiny $i$} (c);
\draw[edge] (b) to (d);
\draw[edge] (c) to (d);
\draw[edge] (d) to (e);
\end{tikzpicture}}}
  \quad\text{,}\quad
 \vcenter{\hbox{%
 \begin{tikzpicture}[x=1.5pt, y=1.5pt]
 \tikzset{vertex/.style = {shape=circle,draw, inner sep=0pt, minimum size=3pt}}
 \tikzset{edge/.style = {->,> = latex'}}
 \node[vertex] (a) at  (-10,20) {};
  \node[vertex] (b) at  (-10,10) {};
 \node[vertex] (c) at  (10,10) {};
 \node[vertex] (d) at  (0,0) {};
 \node[vertex] (e) at  (0,-10) {};
 \draw[edge] (a) to (b);
 \draw[edge] (b) to (d);
 \draw[edge] (c) to (d);
 \draw[edge] (c) to (e);
 \draw[edge] (d) to (e);
 \draw[edge,dashed] (a) to[bend left=40] node [right =-2pt]{\tiny $\alpha$}(d);

 \end{tikzpicture}}}
  \quad\text{,}\quad
 \vcenter{\hbox{%
 \begin{tikzpicture}[x=1.5pt, y=1.5pt]
 \tikzset{vertex/.style = {shape=circle,draw, inner sep=0pt, minimum size=3pt}}
\tikzset{edge/.style = {->,> = latex'}}
\node[vertex] (a) at  (0,20) {};
\node[vertex] (b) at  (-10,10) {};
\node[vertex] (c) at  (10,10) {};
\node[vertex] (d) at  (0,0) {};
\node[vertex] (e) at  (0,-10) {};

\draw[edge] (a) to node [near start,left =-2pt]{\tiny $\beta$} (b);
\draw[edge,dashed] (a) to node[right=-3pt]{\tiny $\gamma$} (d);
\draw[edge] (a) to  (c);
\draw[edge] (b) to (d);
\draw[edge] (c) to (d);
\draw[edge] (d) to (e);
\draw[edge] (c) to (e);

\end{tikzpicture}}}
  \quad\text{.}\quad
\]
Algebras with different graphs are not isomorphic.  Two algebras with the fifth graph and parameters $\alpha, \alpha' $ are isomorphic if and only if
$\alpha' = \alpha $ or $\alpha' = -\alpha$.
 Two algebras with the sixth graph and parameters 
 $(\beta, \gamma) $ and $(\beta', \gamma') $ are isomorphic if and only if $\beta' \in \{\beta, -\beta  \} $ and $\gamma' \in \{\gamma, -\gamma  \}, $
 or $\beta \beta' \in\{1, -1\}$ and  $\gamma'\in \{i\beta^{-1}\gamma, -i\beta^{-1}\gamma  \}$.

 \end{theorem}

   \begin{theorem}
   Let $\cE$ be a  nilpotent evolution algebra of type $[1,1,1,1,1]$, over an algebraically closed field of 
   characteristic not two. Then $\cE$ is isomorphic to an algebra with one of
   the following graphs
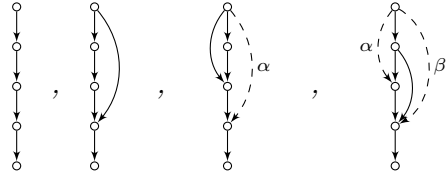
\[
\vcenter{\hbox{%
\begin{tikzpicture}[x=1.5pt, y=1.5pt]

\tikzset{vertex/.style = {shape=circle,draw, inner sep=0pt, minimum size=3pt}}
\tikzset{edge/.style = {->,> = latex'}}
\node[vertex] (a) at  (0,40) {};
\node[vertex] (b) at  (0,30) {};
\node[vertex] (c) at  (0,20) {};
\node[vertex] (d) at  (0,10) {};
\node[vertex] (e) at  (0,0) {};
\draw[edge] (a) to (b);
\draw[edge] (b) to (c);
\draw[edge] (c) to (d);
\draw[edge] (d) to (e);
\end{tikzpicture}}}
  \quad\text{,}\quad
 \vcenter{\hbox{%
 \begin{tikzpicture}[x=1.5pt, y=1.5pt]

\tikzset{vertex/.style = {shape=circle,draw, inner sep=0pt, minimum size=3pt}}
\tikzset{edge/.style = {->,> = latex'}}
\node[vertex] (a) at  (0,40) {};
\node[vertex] (b) at  (0,30) {};
\node[vertex] (c) at  (0,20) {};
\node[vertex] (d) at  (0,10) {};
\node[vertex] (e) at  (0,0) {};
\draw[edge] (a) to (b);
\draw[edge] (b) to (c);
\draw[edge] (c) to (d);
\draw[edge] (d) to (e);
\draw[edge] (a) to[bend left=40] (d);
\end{tikzpicture}}}
  \quad\text{,}\quad
 \vcenter{\hbox{
 \begin{tikzpicture}[x=1.5pt, y=1.5pt]

\tikzset{vertex/.style = {shape=circle,draw, inner sep=0pt, minimum size=3pt}}
\tikzset{edge/.style = {->,> = latex'}}
\node[vertex] (a) at  (0,40) {};
\node[vertex] (b) at  (0,30) {};
\node[vertex] (c) at  (0,20) {};
\node[vertex] (d) at  (0,10) {};
\node[vertex] (e) at  (0,0) {};
\draw[edge] (a) to (b);
\draw[edge] (b) to (c);
\draw[edge] (c) to (d);
\draw[edge] (d) to (e);
\draw[edge,dashed] (a) to[bend left=40] node [right=-2pt]{\tiny $\alpha$}(d);
\draw[edge] (a) to[bend right=40] (c);

\end{tikzpicture}}}
  \quad\text{,}\quad
 \vcenter{\hbox{%
 \begin{tikzpicture}[x=1.5pt, y=1.5pt]

\tikzset{vertex/.style = {shape=circle,draw, inner sep=0pt, minimum size=3pt}}
\tikzset{edge/.style = {->,> = latex'}}
\node[vertex] (a) at  (0,40) {};
\node[vertex] (b) at  (0,30) {};
\node[vertex] (c) at  (0,20) {};
\node[vertex] (d) at  (0,10) {};
\node[vertex] (e) at  (0,0) {};
\draw[edge] (a) to (b);
\draw[edge] (b) to (c);
\draw[edge] (c) to (d);
\draw[edge] (d) to (e);
\draw[edge,dashed] (a) to[bend left=60] node [right=-2pt]{\tiny $\beta$}(d);
\draw[edge,dashed] (a) to[bend right=40] node [left=-2pt]{\tiny $\alpha$}(c);
\draw[edge] (b) to[bend left=40] (d);
\end{tikzpicture}}}
 \]
Algebras with different graphs are not isomorphic.  Two algebras with the third graph and different values of the 
 parameter are not isomorphic.   Two algebras with the fourth graph and parameters $(\alpha,\beta)$ and 
 $(\alpha', \beta')$ are isomorphic if and only if 
 $(\alpha', \beta') = (\alpha,\beta)$ or $(\alpha', \beta') = (-\alpha,-\beta)$.

 \end{theorem}
 
 \medskip
 The results in this section show that the classification in \cite{HA2} misses most of the possibilities.


\bigskip

\section*{Acknowledgments}

The authors are indebted to Professors Hegazi and Abdelwahab for several helpful comments.
 

\end{document}